\documentclass[11pt,reqno]{amsart}
\usepackage{inputenc}
\usepackage{enumerate}
\usepackage{amssymb}
\usepackage{amsmath}
\usepackage{mathrsfs}
\usepackage{amsthm,mathtools}
\usepackage{tikz-cd}
\usepackage{mathpazo}
\usepackage{extarrows}
\usepackage{color,xcolor}
\usepackage{setspace}
\usepackage[colorlinks,linkcolor = blue,anchorcolor = red,urlcolor = blue,citecolor = blue]{hyperref}
\usepackage{graphicx, subfig}
\usepackage[square, comma, sort&compress, numbers]{natbib}
\usepackage{aliascnt}
\usepackage{cleveref} 

\theoremstyle{definition}
\newtheorem{Def}{Definition}[section]

\newaliascnt{Exa}{Def}
\newtheorem{Exa}[Exa]{Examples}
\aliascntresetthe{Exa}
\crefname{Exa}{Example}{Examples}

\newaliascnt{Rem}{Def}
\newtheorem{Rem}[Rem]{Remark}
\aliascntresetthe{Rem}
\crefname{Rem}{Remark}{Remarks}

\newaliascnt{Conv}{Def}

\aliascntresetthe{Conv}
\crefname{Conv}{Convention}{Conventions}

\theoremstyle{plain}
\newaliascnt{Thm}{Def}
\newtheorem{Thm}[Thm]{Theorem}
\aliascntresetthe{Thm}
\crefname{Thm}{Theorem}{Theorems}

\newaliascnt{Lem}{Def}
\newtheorem{Lem}[Lem]{Lemma}
\aliascntresetthe{Lem}
\crefname{Lem}{Lemma}{Lemmas}

\newaliascnt{Pro}{Def}
\newtheorem{Pro}[Pro]{Proposition}
\aliascntresetthe{Pro}
\crefname{Pro}{Proposition}{Propositions}

\newaliascnt{Cor}{Def}
\newtheorem{Cor}[Cor]{Corollary}
\aliascntresetthe{Cor}
\crefname{Cor}{Corollary}{Corollaries}

\newtheorem*{CBC/N}{The relative coarse Baum--Connes/Novikov Conjecture}
\newtheorem*{MCBC/N}{The maximal relative coarse Baum--Connes/Novikov Conjecture}
\newtheorem*{con}{Conjecture}

\usepackage[left=2.6cm, right=2.6cm, top=3cm, bottom=3cm]{geometry}
\usepackage[all]{xy}
\usepackage{bm}

\def\IN{\mathbb N}\def\IR{\mathbb R}\def\IC{\mathbb C}\def\IQ{\mathbb Q}\def\ID{\mathbb D}

\def\B{\mathcal B}\def\E{\mathcal E}\def\G{\mathcal G}\def\K{\mathcal K}\def\L{\mathcal L}\def\H{\mathcal H}\def\S{\mathcal S}\def\M{\mathcal M}

\def\supp{\textup{supp}}

\def\prop{\textup{Prop}}
\def\diam{\textup{diam}}

\def\ox{\otimes}

\def\wt{\widetilde}
\def\ox{\otimes}

\def\Ga{\Gamma}

\def\cal{\curvearrowleft}

\newcommand{\Ad}{\operatorname{Ad}}

\author[L.~Guo]{Liang Guo}
\address[L.~Guo]{Centre for Mathematical Sciences, Lund University, Box 118, Lund, 22100, Sweden}
\email{liang.guo@math.lth.se}

\title{A groupoid approach to the equivariant coarse Baum--Connes conjecture}
\date{\today}

\begin{document}

\begin{abstract}
In this paper, we develop a groupoid approach to the equivariant coarse Baum--Connes conjecture. For a bounded geometry metric space $X$ equipped with a proper, free, and isometric action of a countable discrete group $\Gamma$, we introduce the equivariant coarse groupoid $G(X, \Gamma)$. We prove that the groupoid Baum--Connes conjecture for $G(X, \Gamma)$ with coefficients in $\ell^{\infty}(X,\K)^\Ga$ is equivalent to the equivariant coarse Baum--Connes conjecture for $(X, \Gamma)$ using a localization algebra description of groupoid equivariant $KK$-theory on the Rips complex models. As applications of this framework, we prove that if the space $X$ admits a coarse embedding into Hilbert space (which is not required to be $\Gamma$-equivariant), then the equivariant coarse Novikov conjecture holds for $(X, \Gamma)$, i.e., the assembly map $\mu_{X,\Ga}$ is an injection. We also obtain a new proof of the equivariant coarse Baum--Connes conjecture if $X$ admits an equivariant coarse embedding into Hilbert space.
\end{abstract}

\maketitle

\tableofcontents

\section{Introduction}
The Baum--Connes conjecture \cite{BCH1994} stands as one of the most important problems in non-commutative geometry. For a countable discrete group $\Ga$, it states there is an isomorphism between the equivariant $K$-homology of the classifying space for proper actions of $\Ga$ and the $K$-theory of the reduced group $C^*$-algebra. Its validity has profound applications to geometry and topology, including the Novikov conjecture on the homotopy invariance of higher signatures and obstructions to the existence of positive scalar curvature metrics.
Historically, the Baum--Connes conjecture is rooted in the index theory of compact manifolds, generalizing the classical Atiyah--Singer index theorem. Specifically, it provides a framework to study the higher index of lifted elliptic differential operators on the universal cover $\widetilde{M}$ of a compact manifold $M$. 

Over the past decades, this index-theoretic philosophy has branched out into several major directions. One prominent direction focuses on the index theory of non-compact manifolds. To capture the large-scale behavior of a non-compact space $X$, J.~Roe \cite{HR1993, Roe1996} introduced the equivariant Roe algebra $C^*(X)^\Gamma$ and formulated the \emph{equivariant coarse Baum--Connes conjecture}. Conceptually, this equivariant coarse framework can be viewed as the natural generalization of higher index theory to the universal cover $\widetilde{M}$ of a non-compact manifold $M$. 
Another major direction aims to establish index theory for highly singular spaces, such as foliations and orbit spaces. This gives rise to the study of groupoid $C^*$-algebras and the \emph{groupoid Baum--Connes conjecture}, see \cite{Tu2000} for a survey of the Baum--Connes conjecture for groupoids.

Both directions have witnessed spectacular cutting-edge progress in recent years. In the realm of coarse geometry, G.~Yu proved the coarse Baum--Connes conjecture for spaces admitting coarse embeddings into Hilbert space \cite{Yu2000}. His method inspired a lot of works in the equivariant case \cite{FW2016, DFW2024, FZ2024, GWWY2024}. In the realm of groupoids, J.-L.~Tu \cite{Tu1999} extended N.~Higson and G.~Kasparov's result \cite{HK2001} by proving the groupoid Baum--Connes conjecture for any a-T-menable groupoid.

A striking point connecting these two paths was discovered by Skandalis, Tu, and Yu \cite{STY2002}. They proved that when the group action is trivial, the coarse Baum--Connes conjecture for a metric space $X$ with bounded geometry can actually be realized as a groupoid Baum--Connes conjecture for the coarse groupoid $G(X)$ associated with $X$ with coefficients in $\ell^\infty(X,\K)$. They characterized numerous coarse geometric phenomena in groupoid language. Since then, the groupoid approach has played an important role in tackling various problems surrounding Roe algebras, for example, in the coarse Baum--Connes conjecture \cite{FS2014} and in the ideal classification of Roe algebras \cite{CW2004, WZ2025}.


Despite these successes, extending this groupoid construction to the \emph{equivariant} setting is not a straightforward generalization. Nevertheless, previous achievements strongly suggest that an equivariant analogue of the coarse groupoid would be a powerful tool for tackling related index problems. Motivated by this, the present paper introduces a groupoid model $G(X, \Gamma)$ for the equivariant coarse Baum--Connes conjecture. We construct $G(X, \Gamma)$ as an \'{e}tale, locally compact, Hausdorff groupoid whose unit space is the Stone-\v{C}ech compactification $\beta(X/\Gamma)$.

\begin{Thm}[\Cref{thm:final_equivalence}]\label{thm:intro_main}
Let $X$ be a metric space with bounded geometry, and $\Gamma$ a countable discrete group acting properly, freely, and isometrically on $X$. Then there is a commuting diagram
\[\begin{tikzcd}
\lim\limits_{d \to \infty} K_*^\Gamma(P_d(X))\arrow[r, "\mu_{X,\Gamma}"]\arrow[d, "\cong"']
&
K_*(C^*(X)^\Gamma)\arrow[d, "\cong"]
\\
K_*^{\text{top}}(G(X,\Gamma), A)\arrow[r, "\mu_{G(X,\Gamma)}"]
&
K_*(A \rtimes_r G(X,\Gamma))
\end{tikzcd}\]
where $A=\ell^\infty(X,\K)^\Gamma$.

As a result, the equivariant coarse Baum--Connes conjecture holds for $(X,\Gamma)$ if and only if the groupoid Baum--Connes conjecture holds for $G(X,\Gamma)$ with coefficients in $\ell^\infty(X,\K)^\Gamma$.
\end{Thm}

To establish this equivalence, our strategy is to introduce the language of coarse geometry into the groupoid equivariant setting. Namely, we develop the machinery of \emph{groupoid equivariant Roe algebras} and \emph{groupoid localization algebras} to serve as the analytic foundation of our proof. On the right side, the isomorphism between the $K$-theory of the equivariant Roe algebra and the reduced groupoid $C^*$-algebra is relatively standard, as the $K$-theory of this groupoid Roe algebra naturally recovers the $K$-theory of the reduced groupoid $C^*$-algebra. 
In contrast, establishing the isomorphism on the topological left side, connecting the equivariant $K$-homology of the Rips complex $P_d(X)$ to the topological $K$-theory associated with the classifying space of the groupoid $G(X,\Gamma)$, presents the main technical challenges. The difficulties primarily lie in two aspects. First, there is a geometric difficulty in comparing the two models of classifying spaces, which we address in \Cref{sec: equivalence of assembly maps} (see \Cref{thm:final_equivalence}). Second, one needs to bridge the two different languages of $K$-homology. To achieve this, following the perspectives developed in \cite{TWY2018} and \cite{BenameurRoy2022} for constructing groupoid Roe algebras, we introduce a corresponding \emph{localization algebra} for \'{e}tale groupoids, and prove that for our equivariant coarse groupoid and Rips complex models, its $K$-theory canonically computes the corresponding groupoid equivariant $KK$-theory (\Cref{thm:loc_to_KK}). By resolving these issues on the topological side, we are able to translate the entire groupoid assembly map into a purely coarse geometric framework.

To illustrate the utility of this groupoid machinery, we present two applications. First, we characterize equivariant coarse embeddings into Hilbert space purely in terms of groupoid language. Combining Tu's result on the Baum--Connes conjecture for a-T-menable groupoid \cite{Tu1999}, this perspective yields a conceptual proof of the theorem of Fu and Wang \cite{FW2016}.

\begin{Thm}[\Cref{thm: a-T-menable}]\label{cor:ECBC}
Let $X$ be a bounded geometry metric space equipped with a free, proper, and isometric $\Gamma$-action. The space $X$ admits an equivariant coarse embedding into Hilbert space if and only if its equivariant coarse groupoid $G(X, \Gamma)$ is a-T-menable.
\end{Thm}

As the second application, we show that this framework allows us to weaken the geometric assumptions typically required for the equivariant coarse Novikov conjecture. Specifically, we prove the conjecture for spaces admitting a coarse embedding into Hilbert space, \emph{without requiring the embedding to be $\Gamma$-equivariant}.

\begin{Thm}[\Cref{thm: ECNC for CE spaces}]\label{cor:ECNC}
Let $X$ be a bounded geometry metric space equipped with a free, proper, and isometric $\Gamma$-action. If $X$ admits a coarse embedding into Hilbert space (not necessarily equivariant), then the equivariant coarse Novikov conjecture holds for $(X, \Gamma)$, i.e., the equivariant coarse assembly map
$$\mu_{X,\Ga}: \lim\limits_{d \to \infty} K_*^\Gamma(P_d(X))\to K_*(C^*(X)^\Gamma)$$
is an injection.
\end{Thm}

The proof of this corollary proceeds entirely within our groupoid framework. The key observation is the isomorphism
\[ G(X) \cong G(X, \Gamma) \ltimes \beta X, \]
where the anchor map $\beta X \to \beta(X/\Gamma)$ is extended by the canonical quotient. Because $X$ coarsely embeds into Hilbert space, the coarse groupoid $G(X)$ is a-T-menable \cite{STY2002}. This structural compatibility guarantees that $G(X, \Gamma)$ satisfies the Baum--Connes conjecture with coefficients in $C(\beta X)$. To pass from $C(\beta X)$ to $C(\beta (X/\Ga))$, we employ a measure replacement technique inspired by Higson \cite{Higson2000}. Specifically, we construct probability measures on each fiber and use a Mayer-Vietoris argument in $KK$-theory to complete the proof.

\noindent\textbf{Organization of the paper.} The paper is organized as follows. Section 2 reviews the necessary preliminaries on equivariant Roe algebras and the geometric coarse Baum--Connes conjecture. Section 3 constructs the equivariant coarse groupoid $G(X, \Gamma)$ and explores its Morita equivalence under equivariant coarse equivalence. Section 4 introduces the localization algebra for \'etale groupoids and proves its equivalence with $KK^{\mathcal{G}}$-theory, and proves the main theorem. In Section 5, we use our main theorem to study higher index theory and prove \Cref{cor:ECBC} and \Cref{cor:ECNC}.

\section{Preliminaries on groupoids and equivariant Roe algebras}

Throughout this paper, we shall always assume $X$ is a metric space with bounded geometry, $\Gamma$ is a countable discrete group which acts properly on $X$ from the right. Up to equivariant coarse equivalence, we may assume that the metric $d$ on $X$ takes integer values (e.g., by replacing $d$ with $\lceil d \rceil$). We say the action $X\cal \Gamma$ has \emph{equivariant bounded geometry} if there exists $N>0$ such that $\#\Ga_x\leq N$ for any $x\in X$, where $\Ga_x$ is the stabilizer group of $x$. It is proved in \cite{GWWY2024} that $X\cal \Ga$ has equivariant bounded geometry if and only if there exists a metric space $Z$ with bounded geometry and a proper and \emph{free} $\Ga$-action such that $Z\cal \Gamma$ is equivariantly coarsely equivalent to $X\cal \Ga$. Thus, under equivariant coarse equivalence, we may, without loss of generality, assume that the group action on $X$ is free.

\subsection{Groupoids}

In this subsection, we shall briefly introduce some basic facts of \'etale groupoids. For the definitions of groupoids and their reduced $C^*$-algebras, we refer the reader to \cite{MRW1987, Renault1980}.

We shall first recall the notion of \emph{Morita equivalence} for groupoids. The standard approach utilizes the framework of principal bibundles (or equivalence bispaces), introduced by Muhly, Renault, and Williams \cite{MRW1987}.

\begin{Def}
Let $\G$ be a topological groupoid with unit space $\G^{(0)}$. A \emph{left $\G$-space} is a topological space $Z$ equipped with a continuous, open surjective map $r_Z: Z \to \G^{(0)}$ (which is called a left \emph{anchor} map) and a continuous action map
$$\G \times_{\G^{(0)}} Z \to Z, \quad (g, z) \mapsto g \cdot z,$$
where $\G \times_{\G^{(0)}} Z := \{(g, z) \in \G \times Z \mid s(g) = r_Z(z)\}$ is the \emph{fibered product space} equipped with the subspace topology, such that:
\begin{enumerate}
    \item $r_Z(g \cdot z) = r(g)$ for all $(g, z) \in \G  \times_{\G^{(0)}} Z$;
    \item $g_1 \cdot (g_2 \cdot z) = (g_1 g_2) \cdot z$ whenever $(g_1, g_2) \in \G^{(2)}$ and $(g_2, z) \in \G \times_{\G^{(0)}} Z$;
    \item $r_Z(z) \cdot z = z$ for all $z \in Z$.
\end{enumerate}
A \emph{right $\H$-space} for a groupoid $\H$ is defined analogously, governed by $s_Z: Z \to \H^{(0)}$ and an action map defined on the fibered product $Z \times_{\H^{(0)}} \H := \{(z, h) \in Z \times \H \mid s_Z(z) = r(h)\}$. 
\end{Def}

A left $\G$-action on $Z$ is said to be \emph{free} if $g \cdot z = z$ implies $g \in \G^{(0)}$, and \emph{proper} if the map $\G \times_{\G^{(0)}} Z \to Z \times Z$ given by $(g, z) \mapsto (g \cdot z, z)$ is a proper map. A left (right) $\G$-space $Z$ equipped with a free and proper action is called a \emph{principal left (right) $\G$-space}.

\begin{Def}\label{def:Morita_equivalence}
Two topological groupoids $\G$ and $\H$ are said to be \emph{Morita equivalent} if there exists a topological space $Z$ (called a \emph{$(\G, \H)$-equivalence bispace}) satisfying the following conditions:
\begin{enumerate}
    \item $Z$ is a principal left $\G$-space and a principal right $\H$-space;
    \item The two actions commute, i.e., $(g \cdot z) \cdot h = g \cdot (z \cdot h)$ for all $(g, z) \in \G \times_{\G^{(0)}} Z$ and $(z, h) \in Z \times_{\H^{(0)}} \H$;
    \item The map $Z/\G \to \H^{(0)}$ induced by $s_Z$ is a homeomorphism;
    \item The map $Z/\H \to \G^{(0)}$ induced by $r_Z$ is a homeomorphism.
\end{enumerate}
\end{Def}

The following theorem is proved in the literature. Here, we refer the reader to \cite[Theorem 2.8]{MRW1987} and \cite[Corollary 7.9]{Tu2004}.

\begin{Thm}
If two locally compact, Hausdorff, and \'{e}tale groupoids $\G$ and $\H$ are Morita equivalent, then their reduced groupoid $C^*$-algebras $C^*_r(\G)$ and $C^*_r(\H)$ are strongly Morita equivalent, and hence share the same $K$-theory. \qed
\end{Thm}

The following lemma is perhaps folklore among experts in the theory of groupoid $C^*$-algebras. However, for the sake of completeness and for the reader's convenience, we provide a self-contained proof here, adapted to our specific setting.

\begin{Lem}\label{lem:lsc}
Let $\G$ be an \'etale groupoid and $f \in C_c(\G)$. The map $\G^{(0)} \to \IR$ given by $u \mapsto \|\pi_u(f)\|$ is lower semi-continuous, where $\pi_u$ is the left regular representation of $\G$ on $\ell^2(\G_u)$.
\end{Lem}

\begin{proof}
Fix $u_0 \in \G^{(0)}$ and let $\lambda > 0$ be a real number such that $\|\pi_{u_0}(f)\| > \lambda$. By the definition of the operator norm on the Hilbert space $\ell^2(\G_{u_0})$, there exists a finitely supported unit vector $\xi \in C_c(\G_{u_0})$ such that $\|\pi_{u_0}(f)\xi\|_2 > \lambda$. Let $\supp(\xi) = \{g_1, \dots, g_n\} \subseteq \G_{u_0}$. Because $\G$ is an \'etale groupoid, the source map $s: \G \to \G^{(0)}$ is a local homeomorphism. Thus, we can find mutually disjoint open neighborhoods $B_1, \cdots, B_n$ of $g_1, \cdots, g_n$ in $\G$ such that $g_i \in B_i$ for each $1 \le i \le n$, and $s(B_1) = \dots = s(B_n) = U$ is an open neighborhood of $u_0$ in $\G^{(0)}$. 
For any $u \in U$, with a little abuse of notation, we denote by $g_i(u)$ the unique element in $B_i \cap \G_u$. Notice that $g_i: U\to B_i$ is the inverse map of $s|_{B_i}$ which is clearly a homeomorphism. We can define a continuous field of unit vectors $\xi_u \in C_c(\G_u)$ by setting $\xi_u(g_i(u)) = \xi(g_i)$ and $\xi_u(g) = 0$ for all other $g \in G_u$. Clearly, $\|\xi_u\|_2 = \|\xi\|_2 = 1$ for all $u \in U$.

Consider the vector $\eta_u = \pi_u(f)\xi_u \in \ell^2(\G_u)$. Its value at an element $h \in \G_u$ is given by the groupoid convolution
$$ \eta_u(h) = \sum_{g \in \G_u} f(hg^{-1}) \xi_u(g) = \sum_{i=1}^n f(h g_i(u)^{-1}) \xi(g_i). $$
Because $f$ has compact support and $\G$ is \'etale, the support of $f$ can be covered by finitely many open sets, say $O_1, \dots, O_m$ such that $s|_{O_i}$ is a homeomorphism. If $\eta_u(h) \neq 0$, then $h g_i(u)^{-1} \in \supp(f)$ for some $i$, which implies $h \in \supp(f) \cdot g_i(u)$. Thus, the non-zero values of $\eta_u$ are entirely supported on the finite family of open sets $\{ O_j \cdot g_i(U) \mid 1 \le i \le n, 1 \le j \le m \}$. By definition, there is at most one element $h\in O_j \cdot g_i(U)$ such that $\eta_u(h)\ne 0$. To sum up, $\eta_u$ is also finitely supported.

Since the groupoid multiplication, inversion, the function $f$ and the maps $g_i$ are all continuous, this guarantees that
$$ \|\eta_u\|_2^2 = \sum_{h \in \G_u} |\eta_u(h)|^2 = \sum_{h \in \G_u} \left| \sum_{i=1}^n f(h g_i(u)^{-1}) \xi(g_i) \right|^2 $$
is a finite sum of continuous functions on $U$. Consequently, the map $u \mapsto \|\eta_u\|_2$ is also continuous. Since $\|\eta_{u_0}\|_2 = \|\pi_{u_0}(f)\xi\|_2 > \lambda$, the continuity guarantees the existence of a smaller open neighborhood $V \subseteq U$ containing $u_0$ such that $\|\eta_u\|_2 > \lambda$ for all $u \in V$. 

For any $u \in V$, since $\xi_u$ is a unit vector in $\ell^2(\G_u)$, we have
$$ \|\pi_u(f)\| \ge \|\pi_u(f)\xi_u\|_2 = \|\eta_u\|_2 > \lambda. $$
This proves that the set $\{ u \in \G^{(0)} \mid \|\pi_u(f)\| > \lambda \}$ is open. By definition, the map $u \mapsto \|\pi_u(f)\|$ is lower semi-continuous on $\G^{(0)}$.
\end{proof}

\subsection{Equivariant Roe algebras and the assembly map}
\label{subsec:equivariant_roe}

To formulate the equivariant coarse Baum--Connes conjecture, we first briefly recall the definitions of the equivariant Roe algebra and the equivariant localization algebra. We refer the reader to \cite{FW2016, HIT2020, FZ2024} for more comprehensive details.

\begin{Def}\label{def:ample_module}
Let $X$ be a proper metric space equipped with a proper, free, isometric right $\Gamma$-action. An \emph{equivariant ample $X$-module} is a separable infinite-dimensional Hilbert space $\mathcal{H}_X$ equipped with:
\begin{enumerate}
    \item a faithful, non-degenerate $*$-representation of $C_0(X)$ into $\mathbb{B}(\mathcal{H}_X)$;
    \item a unitary representation $u$ of $\Gamma$ on $\mathcal{H}_X$,
\end{enumerate}
such that
\begin{itemize}
    \item $u_\gamma f u_\gamma^* = \gamma\cdot f$ for all $f \in C_0(X)$ and $\gamma \in \Gamma$;
    \item no non-zero function $f \in C_0(X)$ acts as a compact operator on $\mathcal{H}_X$.
    \item for any finite subgroup $F$ of $\Gamma$ and any $F$-invariant Borel subset $E$ of $X$ there is a Hilbert space $\H_E$ equipped with the trivial representation of $F$ such that $\chi_E \H_X$ is isomorphic to $\H_E\otimes \ell^2(F)$ as $F$-representations.
\end{itemize}
\end{Def}

It is a standard fact that such an equivariant ample module always exists and is unique up to equivariant unitary equivalence. From now on, we fix an equivariant ample $X$-module $\mathcal{H}_X$.

\begin{Def}\label{def:roe_algebra}
Let $T \in \B(\mathcal{H}_X)$ be a bounded linear operator.
\begin{enumerate}
    \item The \emph{propagation} of $T$ is defined to be the infimum of all $R \ge 0$ such that $f T g = 0$ for any pair of functions $f, g \in C_0(X)$ with $d_X(\mathrm{supp}(f), \mathrm{supp}(g)) > R$. If such an $R$ exists, $T$ is said to have \emph{finite propagation}.
    \item $T$ is said to be \emph{locally compact} if both $fT$ and $Tf$ are compact operators for all $f \in C_0(X)$.
    \item $T$ is said to be \emph{$\Gamma$-invariant} if $u_\gamma T u_\gamma^* = T$ for all $\gamma \in \Gamma$.
\end{enumerate}
The \emph{algebraic equivariant Roe algebra}, denoted by $\mathbb{C}[X]^\Gamma$, is defined as the $*$-algebra consisting of all locally compact, $\Gamma$-invariant operators on $\mathcal{H}_X$ with finite propagation. The \emph{equivariant Roe algebra} $C^*(X)^\Gamma$ is defined as the operator norm closure of $\mathbb{C}[X]^\Gamma$ in $\B(\mathcal{H}_X)$. 
\end{Def}

\begin{Def}\label{def:localization_algebra}
The \emph{equivariant localization algebra}, denoted by $C^*_L(X)^\Gamma$, is defined as the $C^*$-algebra generated by all bounded, uniformly norm-continuous functions $f: [0, \infty) \to \IC[X]^\Gamma$ such that $\prop(f(t))\to0$ as $t\to\infty$.
\end{Def}

Evaluation at $t = 0$ provides a canonical $*$-homomorphism:
\[ e: C^*_L(X)^\Gamma \to C^*(X)^\Gamma, \quad e(f) = f(0). \]
This evaluation map naturally induces a homomorphism on the level of $K$-theory:
\[ e_*: K_*(C^*_L(X)^\Gamma) \to K_*(C^*(X)^\Gamma). \]

\begin{Def}[Rips complex]
For each $d\geq 0$, the \emph{Rips complex} of $X$ at scale $d$, denoted $P_d(X)$, consists as a set of all formal sums
$$z=\sum_{z\in X}t_xx$$
such that each $t_x$ is in $[0,1]$, such that $\sum_{x\in X}t_x=1$, and such that the support of $x$ defined by $\supp(z):=\{x\in X\,|\,t_x\ne0\}$ has diameter at most $d$.
\end{Def}

The Rips complex is made into a metric space under the semi-spherical metric \cite{HIT2020}, which makes $P_d(X)$ a proper $\Ga$-space. It is proved in \cite{HIT2020} that the canonical inclusion $i: X\to P_d(X)$ is an equivariant coarse equivalence for each $d\geq0$. Then the equivariant coarse Baum--Connes conjecture states as follows:

\begin{con}[The equivariant coarse Baum--Connes conjecture]
Let $X$ be a proper metric space with bounded geometry equipped with a proper and isometric action of a discrete group $\Gamma$. The equivariant coarse assembly map induced by the evaluation map is an isomorphism, i.e.,
$$\mu_X^{\Ga}:\lim_{d\to\infty}K_*(C^*_L(P_d(X))^{\Ga})\to\lim_{d\to\infty}K_*(C^*(P_d(X))^{\Ga})\cong K_*(C^*(X)^{\Ga})$$
is an isomorphism.
\end{con}

\subsection{A technical lemma}

Throughout this paper, we frequently construct maps on discrete fibered products and extend them to their Stone-\v{C}ech compactifications. While the Stone-\v{C}ech compactification does not commute with fibered products in general, the following lemma identifies a special case where this commutation holds canonically.

\begin{Lem}\label{lem:key lemma}
Let $A, B$, and $C$ be discrete spaces, and let $p_A: A \to C$ and $p_B: B \to C$ be maps. Suppose there exists $N \in \mathbb{N}$ such that $\# p_A^{-1}(c) \le N$ for all $c \in C$. Then the canonical continuous extension of the inclusion $A \times_C B \to \beta A \times_{\beta C} \beta B$ extends to a homeomorphism
\[ \beta(A \times_C B) \to \beta A \times_{\beta C} \beta B. \]
\end{Lem}

\begin{proof}
By Gelfand duality, to prove that the canonical map induces a homeomorphism between the spaces, it suffices to prove that it induces a $C^*$-algebra isomorphism between their corresponding continuous function algebras. The $C^*$-algebra corresponding to the space $\beta(A \times_C B)$ is precisely $\ell^\infty(A \times_C B)$. On the other hand, the $C^*$-algebra corresponding to the fibered product $\beta A \times_{\beta C} \beta B$ is the amalgamated tensor product $C(\beta A) \otimes_{C(\beta C)} C(\beta B)$, which is isomorphic to $\ell^\infty(A) \otimes_{\ell^\infty(C)} \ell^\infty(B)$. Thus, our goal is to show the inclusion
\begin{equation}\label{eq:tensor_fiber_iso}
\ell^\infty(A) \otimes_{\ell^\infty(C)} \ell^\infty(B) \to \ell^\infty(A \times_C B).
\end{equation}
is an isomorphism.

By the assumption that $\# p_A^{-1}(c) \le N$ for all $c \in C$, we can partition the discrete space $A$ into $N$ mutually disjoint subsets
\[ A = A_1 \sqcup A_2 \sqcup \cdots \sqcup A_N, \]
such that the restriction map $p_A|_{A_i}: A_i \to C$ is injective for each $1 \le i \le N$. 
Since this is a finite disjoint union, we have the direct sum decompositions $\ell^\infty(A \times_C B) = \bigoplus_{i=1}^N \ell^\infty(A_i \times_C B)$ and $\ell^\infty(A) \otimes_{\ell^\infty(C)} \ell^\infty(B) \cong \bigoplus_{i=1}^N \left( \ell^\infty(A_i) \otimes_{\ell^\infty(C)} \ell^\infty(B) \right). $

Fix an integer $1 \le i \le N$. Since $p_A|_{A_i}: A_i \to C$ is injective, we can identify $A_i$ with its image $p_A(A_i) \subseteq C$. Under this identification, the algebra $\ell^\infty(A_i)$ is $C^*$-subalgebra of $\ell^\infty(C)$. Consequently, the $i$-th summand in the tensor product simplifies to:
\[ \ell^\infty(A_i) \otimes_{\ell^\infty(C)} \ell^\infty(B) \cong \chi_{A_i} \ell^\infty(C) \otimes_{\ell^\infty(C)} \ell^\infty(B) \cong \chi_{p_B^{-1}(A_i)} \ell^\infty(B)\cong \ell^{\infty}(p_B^{-1}(A_i)). \]
On the other hand, the projection onto the second factor from $A_i \times_C B$ to $B$ is canonically injective since $p_A$ is injective on $A_i$. Moreover, the image of this projection lies in $p_B^{-1}(A_i)$. Therefore, we obtain a natural isomorphism:
\[ \ell^{\infty}(p_B^{-1}(A_i)) \cong \ell^\infty(A_i \times_C B). \]
To sum up, we obtain
\[ \ell^\infty(A) \otimes_{\ell^\infty(C)} \ell^\infty(B) \cong \bigoplus_{i=1}^N \ell^\infty(A_i \times_C B) \cong \ell^\infty(A \times_C B). \]
This shows the isomorphism \eqref{eq:tensor_fiber_iso}.
\end{proof}

\section{A coarse groupoid associated with $X\cal \Ga$}

In this section, we construct a groupoid associated with a system $X\cal\Ga$ and show its relationship with some classic cases. 

Let $X$ be a discrete metric space with bounded geometry. Denote by $\E$ the coarse structure associated with the metric on $X$, i.e.,
$$\E=\left\{E\subseteq X\times X\,\Big|\, \sup_{(x,y)\in E} d(x,y)<\infty\right\}.$$
Let $\Ga$ be a countable discrete group that admits a \emph{proper, free and isometric right} action on $X$. Then $\Ga$ also admits a right action on $X\times X$ by using the diagonal action, i.e.,
$$(x,y)\gamma=(x\gamma,y\gamma).$$
For any discrete metric space $Y$, denote by $\beta(Y)$ the Stone-\v{C}ech compactification of $Y$. It is well-known that $\beta(Y)$ can be realized by the Gelfand spectrum of $\ell^\infty(Y)$. 

\begin{Def}
Let $X$ be a discrete metric space with bounded geometry, $\Ga$ a countable discrete group acting properly, freely and isometrically on $X$. The \emph{equivariant coarse groupoid} is defined by
$$G(X, \Ga)=\bigcup_{R\geq 0}\overline{\Delta_R/\Ga}\subseteq \beta((X\times X)/\Ga),$$
where $\Delta_R = \{(x,y) \in X \times X \mid d(x,y) \le R\}$ is the $R$-entourage.
\end{Def}

In the rest of this subsection, we shall define the groupoid structure on $G(X,\Ga)$. Let us first define the structural maps of the groupoid $G(X, \Ga)$ at the discrete level. Consider the natural projections $p_1, p_2: X \times X \to X$ given by $p_1(x,y) = x$ and $p_2(x,y) = y$. The diagonal action on $X \times X$ satisfies $(x,y)\gamma = (x\gamma, y\gamma)$. Thus, the projections are $\Ga$-equivariant, i.e., $p_i((x,y)\gamma) = p_i(x,y)\gamma$ for $i=1,2$. Thus, these projections descend to well-defined maps on the quotient spaces:
$$r_0, s_0: (X \times X)/\Ga \to X/\Ga$$
defined by $r_0([x,y]) = [x]$ and $s_0([x,y]) = [y]$.

By the universal property of the Stone-\v{C}ech compactification, any map between discrete spaces extends uniquely to a continuous map between their compactifications. Therefore, the continuous maps $r_0$ and $s_0$ extend uniquely to continuous surjective maps:
$$r, s: \beta((X \times X)/\Ga) \to \beta(X/\Ga).$$
By restricting these maps to the subspace $G(X, \Ga)$, we obtain the \emph{range} and \emph{source} maps for the equivariant coarse groupoid.

Similarly, we define the inverse map. The involution $i_0: X \times X \to X \times X$ given by $i_0(x,y) = (y,x)$ is obviously $\Ga$-equivariant and descends to a map $i_0: (X \times X)/\Ga \to (X \times X)/\Ga$, where $i_0([x,y]) = [y,x]$. This extends to a continuous involution on the Stone-\v{C}ech compactification:
$$i: \beta((X \times X)/\Ga) \to \beta((X \times X)/\Ga).$$
Since $i_0(\Delta_R/\Ga) = \Delta_R/\Ga$, its continuous extension restricts to an involution on $G(X, \Ga)$. For any $g \in G(X, \Ga)$, we denote its inverse by $g^{-1} = i(g)$. It is straightforward to check that $r(g^{-1}) = s(g)$ and $s(g^{-1}) = r(g)$.

Finally, the unit map is induced by the diagonal inclusion. The map $u_0: X \to X \times X$, which is clearly $\Ga$-equivariant, given by $u_0(x) = (x,x)$ descends to $u_0: X/\Ga \to (X \times X)/\Ga$, mapping $[x]$ to $[x,x]$. Its continuous extension $u: \beta(X/\Ga) \to \beta((X \times X)/\Ga)$ maps the unit space $\beta(X/\Ga)$ homeomorphically onto $\overline{\Delta_0/\Ga} \subseteq G(X, \Ga)$. We will often identify the unit space $\beta(X/\Ga)$ with its image in $G(X, \Ga)$.

Let $[x,y], [y',z] \in (X \times X)/\Ga$ be two elements such that $s_0([x,y]) = r_0([y',z])$, which means $[y] = [y']$ in $X/\Ga$. Then there exists a \emph{unique} element $\gamma \in \Ga$ such that $y' = y\gamma$. We define their composition as:
$$[x,y] \circ [y',z] := [x, z\gamma^{-1}].$$
Since the group action on $X$ is assumed to be free, it is straightforward to check that the composition is independent of all choices made, and is therefore well-defined.

Let $[x,y] \in \Delta_{R_1}/\Ga$ and $[y',z] \in \Delta_{R_2}/\Ga$. By definition, we can find a representative pair $(x,y)$ such that $d(x,y) \leq R_1$. Since $y' = y\gamma$, the pair $(y', z) = (y\gamma, z)$ is a representative for $[y',z]$. Because the right action of $\Ga$ is isometric, the distance is preserved:
$$d(y, z\gamma^{-1}) = d(y\gamma, z) = d(y', z) \leq R_2.$$
By the triangle inequality on the metric space $X$, we obtain
$$d(x, z\gamma^{-1}) \leq d(x, y) + d(y, z\gamma^{-1}) \leq R_1 + R_2.$$
Therefore, the composed element $[x, z\gamma^{-1}]$ belongs to $\Delta_{R_1+R_2}/\Ga$. As a result, we can extend the composition to the whole equivariant coarse groupoid $G(X, \Ga)$. Let $G_0^{(2)}$ denote the set of composable pairs in the discrete groupoid $(X\times X)/\Ga$. For any fixed $R_1, R_2 \geq 0$, the composition
$$\circ: \left( \Delta_{R_1}/\Ga \times_{X/\Ga} \Delta_{R_2}/\Ga \right) \to \Delta_{R_1+R_2}/\Ga$$
is well-defined. Since $X$ has bounded geometry, the projection $s_0: \Delta_{R_1}/\Ga \to X/\Ga$ has uniformly finite fibers. Thus, by \Cref{lem:key lemma}, this map extends uniquely to a continuous map on the closures in $\beta((X\times X)/\Ga)$. Gluing these maps together for all $R_1, R_2 \geq 0$, we obtain a globally continuous composition map:
$$\circ: G(X, \Ga)^{(2)} \to G(X, \Ga),$$
where $G(X, \Ga)^{(2)} = \{ ([x,y], [y',z]) \in G(X, \Ga) \times G(X, \Ga) \mid [y] = [y'] \}$ is the space of composable pairs in $G(X,\Ga)$. This equips $G(X, \Ga)$ with a well-defined topological groupoid structure.

We shall next prove that the groupoid $G(X, \Ga)$ is \emph{\'etale}. The proof relies heavily on the fact that the $\Ga$-action on $X$ is free and proper and $X$ has bounded geometry. 

\begin{Lem}\label{lem:partition}
For any $R \ge 0$, the discrete space $\Delta_R/\Ga$ can be partitioned into a finite number of mutually disjoint subsets $V_1, V_2, \dots, V_m$ such that both $r_0|_{V_i}$ and $s_0|_{V_i}$ are injective for each $1 \le i \le m$.
\end{Lem}

\begin{proof}
Fix a class $[x] \in X/\Ga$. The preimage $r_0^{-1}([x]) \cap (\Delta_R/\Ga)$ consists of equivalence classes $[x, y]$ such that $d(x,y) \le R$. Suppose $[x, y_1] = [x, y_2]$ are two such classes in $\Delta_R/\Ga$. Since the group action of $\Ga$ on $X$ is free, the map $y \mapsto [x,y]$ is a bijection from the closed ball $B(x, R) \subseteq X$ onto the fiber $r_0^{-1}([x]) \cap (\Delta_R/\Ga)$. Since $X$ has bounded geometry, there exists a uniform constant $N_R \in\IN$ such that $\#B(x, R) \le N_R$ for all $x \in X$. Thus, $\#(r_0^{-1}([x]) \cap (\Delta_R/\Ga)) \le N_R$. By symmetry, the exact same argument shows that $\#(s_0^{-1}([x]) \cap (\Delta_R/\Ga))\leq N_R$.

Consider a graph with vertex set $\Delta_R/\Ga$, where two distinct vertices $v, w$ are connected by an edge if $r_0(v) = r_0(w)$ or $s_0(v) = s_0(w)$. From the above argument, we conclude that any vertex shares the same range with at most $N_R-1$ other vertices, and the same source with at most $N_R-1$ other vertices. Therefore, the degree of any vertex in this graph is strictly bounded by $2N_R-2$. 

By the greedy coloring algorithm, this graph can be colored using at most $2N_R-1$ colors. Let $V_1, \dots, V_m$ be the color classes of $V$, where $m \le 2N_R - 1$. Since no two vertices in the same $V_i$ share an edge, neither $r_0$ nor $s_0$ can map two distinct elements of $V_i$ to the same point in $X/\Ga$. Hence, both $r_0$ and $s_0$ are strictly injective on each $V_i$.
\end{proof}

\begin{Thm}\label{thm:etale}
The equivariant coarse groupoid $G(X, \Ga)$ is an \'etale groupoid.
\end{Thm}

\begin{proof}
It suffices to show that the range map $r: G(X, \Ga) \to \beta(X/\Ga)$ is a local homeomorphism. For any fixed $R \ge 0$, Lemma \ref{lem:partition} yields a finite partition $\Delta_R/\Ga = V_1 \sqcup \dots \sqcup V_m$ consisting of discrete subsets where $r_0$ is injective. For each $i$, the restriction $r_0: V_i \to r_0(V_i)$ is a bijection between discrete spaces, which is trivially a homeomorphism. 

By the universal property of the Stone-\v{C}ech compactification, the closure $\overline{V_i}$ in $\beta((X\times X)/\Ga)$ is naturally homeomorphic to $\beta(V_i)$. The continuous extension $r: \overline{V_i} \to \overline{r_0(V_i)} \subseteq \beta(X/\Ga)$ is therefore a homeomorphism.
It is a standard property of the Stone-\v{C}ech compactification of discrete spaces that the closures of disjoint subsets remain disjoint. Therefore, $\overline{\Delta_R/\Ga} = \overline{V_1} \sqcup \dots \sqcup \overline{V_m}$, meaning each $\overline{V_i}$ is a clopen subset of $\beta((X\times X)/\Ga)$. Since $G(X, \Ga)$ is endowed with the subspace topology, each $\overline{V_i}$ is an open subset of $G(X, \Ga)$. 

Consequently, we have found an open neighborhood $\overline{V_i}$ around every point in $\overline{\Delta_R/\Ga}$ on which $r$ restricts to a homeomorphism. Since $G(X, \Ga) = \bigcup_{R\ge 0} \overline{\Delta_R/\Ga}$, the map $r$ is a local homeomorphism. The identical argument applies to the source map $s$. Thus, $G(X, \Ga)$ is \'etale.
\end{proof}

To illustrate the construction of $G(X,\Ga)$, we provide several examples below. 

\begin{Exa}[Trivial action]
Let $\Ga = \{e\}$ be the trivial group acting on $X$. The quotient space $(X \times X)/\{e\}$ is simply $X \times X$, and $\Delta_R/\{e\} = \Delta_R$. Therefore, the equivariant coarse groupoid reduces to 
$$G(X, \{e\}) = \bigcup_{R\ge 0} \overline{\Delta_R} \subseteq \beta(X \times X).$$
This is exactly the coarse groupoid $G(X)$ introduced by G.~Skandalis, J.~Tu, and G.~Yu in \cite{STY2002}.
\end{Exa}

\begin{Exa}[Group acting on itself]
Let $X = \Ga$ equipped with a proper right-invariant metric, and let $\Ga$ act on itself by right translation. 
Consider the map $\phi: (\Ga \times \Ga)/\Ga \to \Ga$ given by $\phi([g, h]) = gh^{-1}$. It is easy to see that $\phi$ is a well-defined bijection. Under this identification, the entourage $\Delta_R/\Ga$ corresponds to the set 
$$B_R(e) = \{ \gamma \in \Ga \mid |\gamma| \le R \}.$$
Since the metric on $\Ga$ is proper, $B_R(e)$ is a finite set for any $R \ge 0$. The Stone-\v{C}ech compactification of a finite set is just the set itself. Therefore, the closure $\overline{\Delta_R/\Ga}$ in $\beta((\Ga \times \Ga)/\Ga)$ is canonically identified with $B_R(e)$. Taking the union over all $R \ge 0$, we obtain:
$$G(\Ga, \Ga) = \bigcup_{R\ge 0} B_R(e) = \Ga.$$
In this case, the equivariant coarse groupoid degenerates to the group $\Ga$ itself, viewed as a groupoid over a single point.
\end{Exa}

\begin{Exa}[Subgroup action]\label{exa: subgroup action}
Let $G$ be a countable discrete group equipped with a proper right-invariant metric, and $\Ga \le G$ a subgroup acting on $G$ by right translation. 

We define a map $\psi: (G \times G)/\Ga \to G \times (G/\Ga)$ by 
$$\psi([g_1, g_2]) = (g_1 g_2^{-1}, [g_2]).$$
If we choose another representative $(g_1\gamma, g_2\gamma)$ for the orbit, we have $(g_1\gamma)(g_2\gamma)^{-1} = g_1 g_2^{-1}$ and $[g_2\gamma] = [g_2]$ in $G/\Ga$. Hence $\psi$ is a well-defined bijection. 

Under this bijection, the subset $\Delta_R/\Ga$ maps exactly to $B_R(e) \times (G/\Ga)$, where $B_R(e)$ is the $R$-ball in $G$. It is a basic fact that for any finite discrete space $F$ and any space $Y$, there is a natural homeomorphism $\beta(F \times Y) \cong F \times \beta(Y)$. Since $B_R(e)$ is finite, it immediately follows that:
$$\beta(\Delta_R/\Ga) \cong \beta(B_R(e) \times (G/\Ga)) \cong B_R(e) \times \beta(G/\Ga).$$
Taking the union over all $R \ge 0$, we conclude that:
$$G(G, \Ga) \cong \bigcup_{R\ge 0} \left( B_R(e) \times \beta(G/\Ga) \right) = G \times \beta(G/\Ga).$$
One can check that groupoid structure precisely coincides with the transformation groupoid (action groupoid) $G \ltimes \beta(G/\Ga)$ arising from the left action of $G$ on $\beta(G/\Ga)$.
\end{Exa}

\begin{Exa}[Free and cofinite action]
Let $X$ be a metric space with bounded geometry, and suppose $\Ga$ acts freely and cofinitely on $X$ by isometries. Since the action is cofinite, the orbit space $X/\Ga$ is a finite set of $n$ elements. Consequently, the unit space is finite, and $\beta(X/\Ga) = X/\Ga$.

Choose a \emph{fundamental domain} for the action, i.e., select exactly one representative $x_i \in X$ for each of the $n$ orbits. Because the action is free, every element $x \in X$ can be uniquely written as $x = x_i \gamma$ for some $i \in \{1, \dots, n\}$ and $\gamma \in \Ga$. For any orbit $[x, y] \in (X \times X)/\Ga$, we write $x = x_i \gamma_1$ and $y = x_j \gamma_2$. Using the diagonal right action, we have
$$[x, y] = [x_i \gamma_1, x_j \gamma_2] = [x_i \gamma_1 \gamma_2^{-1}, x_j].$$
We then define $\Phi: G(X,\Ga)\to M_n(\Ga)$ by $[x,y] \mapsto (i, \gamma_1 \gamma_2^{-1}, j)$, where $M_n(\Ga)$ is the matrix groupoid. This map is clearly a well-defined bijection. One can verify that $\Phi$ is a groupoid isomorphism.

This explicit isomorphism shows that $G(X, \Ga)$ is Morita equivalent to the group $\Ga$, reflecting the well-known geometric fact that the equivariant Roe algebra of a cofinite action is stably isomorphic to the reduced group $C^*$-algebra $C^*_r(\Ga)$.
\end{Exa}

\subsection{Equivariant Roe algebras and groupoid algebras}

In this subsection, we prove the main theorem of this section, i.e., the reduced groupoid $C^*$-algebra of $G(X, \Ga)$ is canonically isomorphic to the uniform equivariant Roe algebra $C^*_u(X)^\Ga$. Recall that the algebraic equivariant uniform Roe algebra $\IC_u[X]^\Ga$ is the $*$-subalgebra of $\B(\ell^2(X))$ consisting of all finite propagation operators $T$ that are $\Ga$-invariant, i.e., $u_\gamma T u_\gamma^* = T$ and $u$ is the right regular representation. Its operator norm closure in $\B(\ell^2(X))$ is the equivariant uniform Roe algebra $C^*_u(X)^\Ga$.

\begin{Thm}\label{thm: groupoid to Roe}
Let $X$ be a metric space with bounded geometry, and $\Ga$ a discrete group acting freely, properly, and isometrically on $X$. There is a canonical $*$-isomorphism
$$ C^*_r(G(X, \Ga)) \cong C^*_u(X)^\Ga. $$
\end{Thm}

\begin{proof}
We divide the proof into two steps.

\noindent\emph{Step 1: Algebraic isomorphism.} 
By the construction of the Stone-\v{C}ech compactification, $G(X, \Ga)$ is the union of clopen sets $\overline{\Delta_R/\Ga}$ for $R \ge 0$. A compactly supported continuous function $f \in C_c(G(X, \Ga))$ is therefore supported on some $\overline{\Delta_R/\Ga}$.

We define a map $\Phi: C_c(G(X, \Ga)) \to \mathcal{B}(\ell^2(X))$ by associating $f$ to an operator $T_f$ whose matrix coefficients are given by:
$$ \langle \delta_x, T_f \delta_y \rangle = T_f(x, y) := f([x, y]), \quad \text{for all } x, y \in X. $$
Since $f$ is supported on $\overline{\Delta_R/\Ga}$, $T_f(x,y) = 0$ whenever $d(x,y) > R$, meaning $T_f$ has finite propagation. Moreover, the definition depends only on the orbit $[x,y]$, so $T_f(x\gamma, y\gamma) = f([x\gamma, y\gamma]) = f([x,y]) = T_f(x,y)$, which implies that $T_f$ is $\Ga$-invariant. Because $X$ has bounded geometry, bounded matrix coefficients with finite propagation define a bounded operator. Thus, $T_f \in \IC_u[X]^\Ga$.

To see that $\Phi$ preserves the algebra structure, let $f, g \in C_c(G(X, \Ga))$. Their convolution evaluated at a discrete point $[x,z]$ is given by
$$(f * g)([x, z]) = \sum_{y \in X} f([x,y]) g([y,z]) = \sum_{y \in X} T_f(x, y) T_g(y, z) = (T_f T_g)(x,z).$$
The involution is preserved similarly. Furthermore, $\Phi$ is bijective because any operator in $\IC_u[X]^\Ga$ uniquely defines a bounded function on some $\Delta_R/\Ga$, which extends continuously and uniquely to $\overline{\Delta_R/\Ga}$ by the universal property of the Stone-\v{C}ech compactification. Thus, $\Phi$ is a $*$-isomorphism from $C_c(G(X, \Ga))$ onto $\IC_u[X]^\Ga$.

\noindent\emph{Step 2: Comparison of $C^*$-norms.} 
It suffices to show that the reduced groupoid norm of $f \in C_c(G(X, \Ga))$ equals the operator norm of $T_f$ on $\ell^2(X)$. The reduced groupoid norm is defined as
$$ \|f\|_r = \sup_{\omega \in \beta(X/\Ga)} \|\pi_\omega(f)\|, $$
where $\pi_\omega$ is the left regular representation of $G(X, \Ga)$ on $\ell^2(s^{-1}(\omega))$.

By \Cref{lem:lsc}, since $G(X, \Ga)$ is an \'etale groupoid, the norm map $\omega \mapsto \|\pi_\omega(f)\|$ is lower semi-continuous on the unit space $G^{(0)} = \beta(X/\Ga)$. The unit space contains $X/\Ga$ as a dense open subset. The supremum of a lower semi-continuous function over a dense subset coincides with the supremum over the entire space. Therefore, we have that
$$ \|f\|_r = \sup_{[x] \in X/\Ga} \|\pi_{[x]}(f)\|. $$

For a discrete point $[x] \in X/\Ga$, its source fiber $s^{-1}([x])=\{[y,x]\in (X \times X)/\Ga \mid y \in X \}$. Because the $\Ga$-action is free, the map $U: \ell^2(s^{-1}([x])) \to \ell^2(X)$ given by $U(\delta_{[y,x]}) = \delta_y$ is a unitary operator. By the definition of groupoid convolution,
$$ \pi_{[x]}(f) \delta_{[y,x]} = \sum_{[z, y] \in (X\times X)/\Ga} f([z, y]) \delta_{[z, y] \circ [y, x]} = \sum_{z \in X} f([z, y]) \delta_{[z, x]}. $$
Conjugating this representation by the unitary $U$, we have that
$$ (U \pi_{[x]}(f) U^*) \delta_y = U \Big( \sum_{z \in X} f([z, y]) \delta_{[z, x]} \Big) = \sum_{z \in X} T_f(z, y) \delta_z = T_f \delta_y. $$
This shows that for every $[x] \in X/\Ga$, the left regular representation $\pi_{[x]}$ is unitarily equivalent to the operator $T_f$ acting on $\ell^2(X)$. 

Consequently, taking the supremum over all $[x] \in X/\Ga$, we obtain
$$ \|f\|_r = \sup_{[x] \in X/\Ga} \|T_f\|_{\mathcal{B}(\ell^2(X))} = \|T_f\|_{\mathcal{B}(\ell^2(X))}. $$
Therefore, the map $\Phi$ extends to a $C^*$-isomorphism between $C^*_r(G(X, \Ga))$ and $C^*_u(X)^\Ga$.
\end{proof}

An analogous argument holds for the standard equivariant Roe algebra $C^*(X)^\Ga$, which is represented on $H_X \cong \ell^2(X) \ox\ell^2(\Ga)\ox H$ for an infinite-dimensional separable Hilbert space $H$. Rather than a scalar-valued groupoid $C^*$-algebra, one considers the groupoid crossed product $G(X, \Ga)\ltimes_r \ell^\infty(X, \mathcal{K})^\Ga$, where $\K=\K(\ell^2(\Ga)\ox H)$ equipped with the conjugate action induced by the right action of $\Ga$ on $\ell^2(\Ga)$. Here $A = \ell^\infty(X, \K)^\Ga$ is naturally a $C(\beta(X/\Ga))$-algebra since $\ell^\infty(X/\Ga) \cong C(\beta(X/\Ga))$ acts centrally on $A$. For each clopen bisection $\overline{V} \subseteq \overline{\Delta_R/\Ga}$ as in \Cref{lem:partition}, the bijection $V \subseteq (X \times X)/\Ga$ induces a $\Ga$-equivariant partial translation between the corresponding orbits in $X$, which defines an isomorphism $\alpha_{\overline{V}} : C(s(\overline{V}))A \to C(r(\overline{V}))A$ and hence a continuous $G(X,\Ga)$-action on $A$. The proof proceeds mutatis mutandis by replacing scalar-valued kernels with $\mathcal{K}$-valued kernels.

\begin{Cor}\label{cor: equivariant Roe and groupoid}
Let $X$ be a metric space with bounded geometry, and $\Ga$ a countable discrete group acting freely, properly, and isometrically on $X$. There is a canonical $C^*$-isomorphism
$$G(X, \Ga)\ltimes_r \ell^\infty(X, \mathcal{K})^\Ga \cong C^*(X)^\Ga. \qed$$
\end{Cor}

As a final result of this section, we prove the Morita equivalence of the equivariant coarse groupoid under equivariant coarse equivalence. 

\begin{Thm}\label{thm: morita coarse}
Let $X$ and $Y$ be bounded geometry metric spaces equipped with free, proper, and isometric right $\Ga$-actions. If $X$ and $Y$ are equivariantly coarsely equivalent, then their equivariant coarse groupoids $G(X, \Ga)$ and $G(Y, \Ga)$ are Morita equivalent.
\end{Thm}

\begin{proof}
Let $f: X \to Y$ be an equivariant coarse equivalence. For each integer $R \ge 0$, define the discrete space
$$ W_R = \{ (x, y) \in X \times Y \mid d_Y(f(x), y) \le R \}. $$
Because $f$ is $\Ga$-equivariant and the action on $Y$ is isometric, $W_R$ is invariant under the diagonal right $\Ga$-action on $X \times Y$. We pass to the quotient space $W_R/\Ga \subseteq (X \times Y)/\Ga$. We then define our bispace $Z$ as the closure in the Stone-\v{C}ech compactification:
$$ Z = \bigcup_{R \ge 0} \overline{W_R/\Ga} \subseteq \beta((X \times Y)/\Ga). $$

Define the left anchor map $r_Z: Z \to \beta(X/\Ga)$ and the right anchor map $s_Z: Z \to \beta(Y/\Ga)$ as the continuous extensions of the canonical projections $p_X([x,y]) = [x]$ and $p_Y([x,y]) = [y]$, respectively. Since $Y$ has bounded geometry and $f$ is coarsely proper on the bounded geometry space $X$, both projections $p_X: W_R/\Ga \to X/\Ga$ and $p_Y: W_R/\Ga \to Y/\Ga$ have uniformly finite fibers for each $R \ge 0$. The left $G(X, \Ga)$-action on $Z$ is defined on the discrete dense subsets by
$$ [x_1, x_2] \cdot [x_2, y] = [x_1, y]. $$
To see that $[x_1, y] \in Z$, note that if $[x_1, x_2] \in \Delta_{R'}^X/\Ga$ and $[x_2, y] \in W_R/\Ga$, then $d_X(x_1, x_2) \le R'$ and $d_Y(f(x_2), y) \le R$. Since $f$ is a coarse equivalence, $d_Y(f(x_1), f(x_2)) \le S$ for some $S > 0$ determined by $R'$. The triangle inequality yields $d_Y(f(x_1), y) \le S + R$, meaning $[x_1, y] \in W_{S+R}/\Ga \subset Z$. Thus, the left action is well-defined and extends continuously to $Z$ by \Cref{lem:key lemma}. Similarly, the right $G(Y, \Ga)$-action defined by $[x, y_1] \cdot [y_1, y_2] = [x, y_2]$ maps $W_R/\Ga \times_{Y/\Ga} \Delta_{R'}^Y/\Ga$ into $W_{R+R'}/\Ga$, which extends continuously to $Z$ by \Cref{lem:key lemma} and commutes with the left action.

To see the left-action is free and proper and $s_Z$ induces a homeomorphism between $Z/G(X, \Ga)$ and  $\beta(Y/\Ga)$, consider any two elements $z_1 = [x_1, y_1]$ and $z_2 = [x_2, y_2]$ in $W_R/\Ga$ with $p_Y(z_1) = p_Y(z_2)$ in $Y/\Ga$. Write $y_2 = y_1\gamma$ for some $\gamma \in \Ga$, so that $z_2 = [x_2\gamma^{-1}, y_1]$. By definition, we have that $d_Y(f(x_1), y_1) \le R$ and $d_Y(f(x_2\gamma^{-1}), y_1) = d_Y(f(x_2), y_2) \le R$. This implies $d_Y(f(x_1), f(x_2\gamma^{-1})) \le 2R$. Because $f$ is a coarse equivalence, there exists $M > 0$ depending only on $R$ such that $d_X(x_1, x_2\gamma^{-1}) \le M$. Thus, setting
$$ \delta_L(z_1, z_2) := [x_1, x_2\gamma^{-1}] \in \Delta_M^X/\Ga $$
defines a discrete map $\delta_L: W_R/\Ga \times_{Y/\Ga} W_R/\Ga \to \Delta_M^X/\Ga$, which extends to $\overline{W_R/\Ga} \times_{\beta(Y/\Ga)} \overline{W_R/\Ga} \to \overline{\Delta_M^X/\Ga}$ for each $R \ge 0$ by \Cref{lem:key lemma}, and hence defines a continuous map $\delta_L: Z \times_{\beta(Y/\Ga)} Z \to G(X, \Ga)$. By continuity from the dense discrete subsets, $\delta_L$ satisfies $\delta_L(z_1, z_2) \cdot z_2 = z_1$ and $\delta_L(g \cdot z_2, z_2) = g$ for all $(z_1, z_2) \in Z \times_{\beta(Y/\Ga)} Z$ and $(g, z_2) \in G(X, \Ga) \times_{\beta(X/\Ga)} Z$. Consequently, the action map $(g, z) \mapsto (g \cdot z, z)$ is a homeomorphism onto the closed subset $Z \times_{\beta(Y/\Ga)} Z \subseteq Z \times Z$ with continuous inverse $(z_1, z_2) \mapsto (\delta_L(z_1, z_2), z_2)$, which proves that the left $G(X, \Ga)$-action on $Z$ is both free and proper. Moreover, since $f$ is a coarse equivalence, $p_Y(W_C/\Ga) = Y/\Ga$ for some large $C > 0$, so $s_Z: Z \to \beta(Y/\Ga)$ is a continuous surjection whose fibers are precisely the $G(X, \Ga)$-orbits by the identity $\delta_L(z_1, z_2) \cdot z_2 = z_1$. Hence $s_Z$ induces a homeomorphism $Z/G(X, \Ga) \cong \beta(Y/\Ga)$.

Symmetrically, one can also show that the right $G(Y, \Ga)$-action is free and proper and that $r_Z$ induces a homeomorphism $Z/G(Y, \Ga) \cong \beta(X/\Ga)$. This finishes the proof.
\end{proof}

Recall from the preceding section that a $\Ga$-space $X$ has equivariant bounded geometry if and only if there exists a metric space $Y$ equipped with a free and proper isometric $\Ga$-action such that $X$ and $Y$ are equivariantly coarsely equivalent. For any $\Ga$-space $X$ with equivariant bounded geometry, even if its action is not free, we can formally define its equivariant coarse groupoid as the groupoid $G(Y, \Ga)$, where $Y$ is any free model that is equivariantly coarsely equivalent to $X$. \Cref{thm: morita coarse} guarantees that the resulting groupoid is completely independent of the choice of the free space $Y$ up to Morita equivalence.

\section{Localization algebras and $KK^{\G}$-theory}
\label{sec:localization_kk}

In this section, we introduce a localization algebra description of the groupoid equivariant $KK$-theory. To do this, we first recall the geometric and analytic framework of metric $\G$-spaces and their associated dual Roe algebras, as systematically developed by Benameur and Roy \cite{BenameurRoy2022}. We then bridge their Paschke-Higson duality with the localization algebra framework.

\subsection{Metric $\G$-spaces and groupoid equivariant Roe algebras with coefficients}

Throughout this section, we will always assume $\G$ to be a locally compact, Hausdorff, \'{e}tale groupoid with unit space $\G^{(0)}$.

\begin{Def}\label{def:proper G-space}
A \emph{metric $\G$-space} is a locally compact, Hausdorff, proper $\G$-space $Y$ with an anchor map $\rho: Y \to \G^{(0)}$, along with a continuous family of metrics $\{d_x\}_{x \in \G^{(0)}}$ on the fibers $Y_x = \rho^{-1}(x)$ with respect to the subspace topology on each $Y_x$ such that
\begin{itemize}
    \item[(1)] The metric is \emph{proper} in the sense that for any compact subset $K \subseteq Y$ and any $R \ge 0$, its closed fiberwise $R$-neighborhood $\{y \in Y \mid \exists z \in K \cap Y_{\rho(y)} \text{ such that } d_{\rho(y)}(y, z) \le R\}$ is compact.
    \item[(2)] The $\G$-action is \emph{fiberwise isometric}, i.e., for any $y_1, y_2 \in Y_{s(g)}$,
    $$d_{r(g)}(gy_1, gy_2) = d_{s(g)}(y_1, y_2).$$
\end{itemize}
\end{Def}

A canonical example of a metric $\G$-space is the groupoid $\G$ acting on itself by left translation. The anchor map is the range map $r: \G \to \G^{(0)}$, and the fibers are $\G_x = r^{-1}(x)$. A \emph{length function} (see \cite{OY2019, MW2020}) on $\G$ is a function $\ell:\G\to\IR_+$ such that
\begin{itemize}
    \item[(1)] $\ell(g)=0$ if and only if $g\in\G^{(0)}$;
    \item[(2)] $\ell(g)=\ell(g^{-1})$ for all $g\in\G$;
    \item[(3)] $\ell(g_1g_2)\leq \ell(g_1)+\ell(g_2)$ for any $g_1, g_2\in\G$ with $s(g_1)=r(g_2)$.
\end{itemize}
This length function $\ell$ is \emph{proper} if $r: \ell^{-1}([0,M])\to\G^{(0)}$ is proper for any $M\geq 0$. A proper length function always exists on a $\sigma$-compact groupoid $\G$. With a proper length function, we define the metric on $\G_x$ to be
$$d_x(g_1,g_2)=\ell(g_1^{-1}g_2)$$
for each $x\in\G^{(0)}$. It is direct to check that $\G$ is a proper, free, metric $\G$-space under this family of metrics $\{d_x\}_{x\in\G^{(0)}}$.

Let $A$ be a $C^*$-algebra. The algebra $A$ is called a \emph{$C_0(\G^{(0)})$-algebra} if there exists a non-degenerate $*$-homomorphism from $C_0(\G^{(0)})$ into the center of the multiplier algebra $\mathcal{M}(A)$. For each $x \in \G^{(0)}$, the fiber $A_x$ is defined as the quotient $C^*$-algebra
$$A_x=\frac{A} { C_0(\G^{(0)} \setminus \{x\})A}.$$
Thus $A$ can be viewed as a continuous family of fiber $C^*$-algebras $(A_x)_{x\in \G^{(0)}}$. For a continuous map $f:X\to \G^{(0)}$, we define the \emph{pullback} of $A$ by $f$ via the spatial tensor product
$$f^*A=A\ox_{C_0(\G^{(0)})}C_0(X).$$
More precisely, it is the quotient of $A\ox C_0(X)$ modulo the relation $a\phi\ox \psi\sim a\ox f^*(\phi)\psi$, where $\phi\in C_0(\G^{(0)})$ and $\psi\in C_0(X)$. It is straightforward to verify that for any $x\in X$, the fiber algebra $(f^*A)_{x}$ is canonically isomorphic to $A_{f(x)}$.

A \emph{$\G$-$C^*$-algebra} is a $C_0(\G^{(0)})$-algebra $A$ equipped with a continuous $\G$-action. Specifically, the action is given by a globally continuous isomorphism $\alpha: s^*A \to r^*A$ of $C_0(\G)$-algebras. This induces a family of $*$-isomorphisms $\alpha_g: A_{s(g)} \to A_{r(g)}$ on the fibers such that the cocycle condition $\alpha_{gh} = \alpha_g \circ \alpha_h$ holds for any composable pair $(g, h) \in \G^{(2)}$.

To incorporate the coefficient algebra $A$, we use Hilbert modules over $A$. Let $\E$ be a Hilbert $A$-module. Because $A$ is a $C_0(\G^{(0)})$-algebra, $\E$ is naturally a $C_0(\G^{(0)})$-module. Analogous to the algebra case, the fiber of $\E$ at $x \in \G^{(0)}$ is the quotient space
$$\E_x=\frac{\E}{C_0(\G^{(0)}\backslash\{x\})\cdot \E}.$$
This fiber $\E_x$ naturally carries the structure of a Hilbert $A_x$-module. Similarly, for a continuous map $f:X\to \G^{(0)}$, the \emph{pullback} $f^*\E = \E\ox_{C_0(\G^{(0)}),f}C_0(X)$ naturally becomes a Hilbert $f^*A$-module, and its fiber at $x \in X$ is canonically isomorphic to $\E_{f(x)}$.

A \emph{$\G$-action} on a Hilbert $A$-module $\E$ is an isomorphism of Hilbert $s^*A$-modules $V: s^*\mathcal{E} \xrightarrow{\sim} r^*\mathcal{E}$. This means
\begin{itemize}
    \item[(1)] $V$ restricts to a unitary isomorphism of Hilbert modules $V_g: \E_{s(g)}\to \E_{r(g)}$ that is compatible with the algebra isomorphism $\alpha_g: A_{s(g)} \to A_{r(g)}$, meaning $V_g(\xi \cdot a) = V_g(\xi) \cdot \alpha_g(a)$ for any $\xi \in \E_{s(g)}$ and $a \in A_{s(g)}$.
    \item[(2)] The family $\{V_g\}_{g\in\G}$ satisfies $V_{gh} = V_g V_h$ over composable pairs in $\G$.
\end{itemize}
We call such an $\E$ a \emph{Hilbert $\G$-$A$-module}. Let $\mathcal{L}_{A}(\mathcal{E})$ and $\mathcal{K}_{A}(\mathcal{E})$ denote the $C^*$-algebras of adjointable operators and generalized compact operators on $\mathcal{E}$, respectively.

\begin{Def}\label{def:admissible_rep}
Let $Y$ be a metric $\G$-space. An \emph{admissible representation} of $C_0(Y)$ on a Hilbert $\G$-$A$-module $\mathcal{E}$ is a $*$-homomorphism $\pi: C_0(Y) \to \mathcal{L}_{A}(\mathcal{E})$ satisfying:
\begin{itemize}
    \item[(1)] \emph{non-degenerate:} The linear span of $\pi(C_0(Y))\mathcal{E}$ is dense in $\mathcal{E}$.
    \item[(2)] \emph{$C_0(\G^{(0)})$-linear:} For any $f \in C_0(Y)$ and $\phi \in C_0(\G^{(0)})$, $\pi(f \cdot \rho^*(\phi)) = \pi(f)\phi$.
    \item[(3)] \emph{equivariant:} The representation strictly intertwines the $\G$-actions on $Y$ and $\mathcal{E}$. Fiberwise, this means $V_g \pi_{s(g)}(f) V_g^* = \pi_{r(g)}(g \cdot f)$.
    \item[(4)] \emph{ample:} For any non-zero function $f \in C_0(Y)$, the operator $\pi(f)$ is not a generalized compact operator, i.e., $\pi(f) \notin \mathcal{K}_{A}(\mathcal{E})$.
\end{itemize}
\end{Def}

We now show an example of the standard admissible module for proper $\G$-spaces with a \emph{full $\rho$-system}.

\begin{Def}\label{def:rho system}
Let $Y$ be a metric $\G$-space with anchor map $\rho: Y \to \G^{(0)}$. A \emph{full $\rho$-system} on $Y$ is a $\G$-equivariant continuous family of strictly positive Radon measures $\{\mu_x\}_{x \in \G^{(0)}}$ on the fibers $Y_x$, such that for any $f\in C_c(Y)$, the function
$$x\mapsto \int_{Y_x}f(y)d\mu_x(y)$$
is continuous on $\G^{(0)}$.
\end{Def}

A full $\rho$-system on $\G$ itself is called a \emph{Haar system}, which always exists if $\G$ is proper and \'{e}tale. Assume that the metric $\G$-space $Y$ admits a full $\rho$-system $\{\mu_x\}_{x \in \G^{(0)}}$. Let $C_c(\G)$ be endowed with the natural $C_0(\G^{(0)})$-module structure $(\xi \cdot \phi)(g) = \xi(g)\phi(r(g))$ and the $C_0(\G^{(0)})$-valued inner product:
$$ \langle \xi, \eta \rangle_{\ell^2(\G)}(x) = \sum_{g \in r^{-1}(x)} \overline{\xi(g)}\eta(g). $$
Its completion yields a standard Hilbert $C_0(\G^{(0)})$-module, denoted by $\ell^2(\G)$, whose fiber over $x$ is exactly $\ell^2(r^{-1}(x))$. The groupoid $\G$ acts on $\ell^2(\G)$ by isometries $V^\G_\gamma: \ell^2(r^{-1}(s(\gamma))) \to \ell^2(r^{-1}(r(\gamma)))$ given by $(V^\G_\gamma \xi)(g) = \xi(\gamma^{-1}g)$.

We define the module $L^2(Y, \mu)$ as the completion of $C_c(Y)$ with respect to the $C_0(\G^{(0)})$-valued inner product:
$$ \langle \xi, \eta \rangle_{L^2(Y)}(x) = \int_{Y_x} \overline{\xi(y)}\eta(y) d\mu_x(y). $$
The $C^*$-algebra $C_0(Y)$ naturally acts on $L^2(Y, \mu)$ by pointwise multiplication $M_f(\xi)(y) = f(y)\xi(y)$, and $\G$ acts isometrically by $(V^Y_\gamma \xi)(y) = \xi(\gamma^{-1}y)$.

Define the spatial internal tensor product of these Hilbert modules:
\begin{equation}\label{eq: admissible rep}
\mathcal{E}_Y = L^2(Y, \mu) \otimes_{C_0(\G^{(0)})} \ell^2(\G) \otimes_{C_0(\G^{(0)})} A \otimes H_0,
\end{equation}
where $H_0=\ell^2(\IN)$ is a fixed separable infinite-dimensional Hilbert space. Moreover,
\begin{itemize}
    \item the non-degenerate $C_0(Y)$-action is given by $\pi(f) = M_f \otimes 1 \otimes 1 \otimes 1$;
    \item the $\G$-action is the diagonal action $V_\gamma = V^Y_\gamma \otimes V^\G_\gamma \otimes \alpha_\gamma \otimes 1$.
\end{itemize}
The presence of ${H}_0$ guarantees that $\pi(f)$ is never a compact operator in $\mathcal{K}_A(\E_Y)$ for any $f \neq 0$. This $\E_Y$ serves as the canonical universal model for admissible Hilbert $\G$-$A$-modules over any metric $\G$-space $Y$ equipped with a full $\rho$-system.

\begin{Rem}\label{rem:rho_system_existence}
By a classic result of Blanchard (see \cite{Blanchard96, Williams2016}), a full $\rho$-system always exists when $Y$ is $\G$-proper and second countable. For equivariant coarse groupoids (as they involve the Stone-\v{C}ech compactification), the existence of a full $\rho$-system is not automatically guaranteed by abstract existence theorems. However, we shall only study the case when $Y$ is constructed out of Rips complexes, whose full $\rho$-system can be written explicitly, see \Cref{lem:K_homology_iso}. Therefore, throughout this paper, we shall safely assume the existence of a full $\rho$-system.
\end{Rem}

Throughout the remainder of this section, we take $\mathcal{E} = \mathcal{E}_Y$ to be the canonical admissible Hilbert $\mathcal{G}$-$A$-module defined in \eqref{eq: admissible rep} . For any subsets $A, B \subseteq Y$, we set 
$$d_Y(A, B) = \inf_{x \in \G^{(0)}} d_x(A \cap Y_x, B \cap Y_x).$$
Since $T \in \mathcal{L}_{A}(\E)$ is $C_0(\G^{(0)})$-linear, it preserves the submodule $C_0(\G^{(0)}\backslash\{x\}) \cdot \E$. Thus, it induces a bounded adjointable operator $T_x$ on the fiber Hilbert $A_x$-module $\E_x$ for any $x\in\G^{(0)}$.

\begin{Def}
Let $T \in \mathcal{L}_{A}(\mathcal{E})$ be an adjointable operator.
\begin{itemize}
    \item[(1)] $T$ has \emph{finite propagation} if there exists $R \ge 0$ such that for any $f, g \in C_0(Y)$ satisfying $d_Y(\supp(f), \supp(g)) > R$, we have $f T g = 0$.
    \item[(2)] $T$ is \emph{locally compact} if for all $f \in C_0(Y)$, $fT$ and $Tf$ belong to $\mathcal{K}_{A}(\mathcal{E})$.
    \item[(3)] $T$ is \emph{$\G$-invariant} if it commutes with the $\G$-action on $\mathcal{E}$, meaning $V_g T_{s(g)} = T_{r(g)} V_g$ for all $g \in \G$.
    \item[(4)] $T$ is \emph{pseudo-local} if $[T, f] \in \mathcal{K}_{A}(\E)$ for all $f \in C_0(Y)$. Equivalently, for any $f,g\in C_c(Y)$ with $d_Y(\supp(f), \supp(g)) > 0$, we have $fTg$ is compact.
\end{itemize}
\end{Def}

Following Benameur and Roy \cite{BenameurRoy2022}, we define the equivariant Roe algebra and dual algebra with coefficients in $A$:

\begin{Def}
Let $Y$ be a metric $\G$-space. 
\begin{itemize}
    \item[(1)] The \emph{$\G$-equivariant Roe algebra with coefficients in $A$}, denoted by $C^*(Y, A)^{\G}$, is the norm closure in $\mathcal{L}_{A}(\E)$ of the $*$-algebra, denoted by $\IC[Y,A]^{\G}$, of all $\G$-equivariant, locally compact operators with finite propagation.
    \item[(2)] The \emph{$\G$-equivariant dual algebra with coefficients in $A$}, denoted by $D^*(Y, A)^\G$, is the norm closure of the $*$-algebra, denoted by $\ID[Y,A]^{\G}$, of all $\G$-equivariant, pseudo-local operators with finite propagation.
\end{itemize}
\end{Def}

It is proved in \cite{BenameurRoy2022} that $C^*(Y, A)^\G$ forms a closed two-sided ideal in $D^*(Y, A)^\G$. We thus define the Paschke quotient algebra by $Q^*(Y, A)^\G = D^*(Y, A)^\G / C^*(Y, A)^\G$, yielding the Higson-Roe exact sequence:
$$ 0 \to C^*(Y, A)^\G \to D^*(Y, A)^\G \to Q^*(Y, A)^\G \to 0. $$
The $K$-theory boundary map is exactly $\partial: K_{*+1}(Q^*(Y, A)^\G) \to K_{*}(C^*(Y, A)^\G)$. 
The Paschke-Higson duality naturally extends to coefficients in $A$, providing a homomorphism
$$\mathcal{P}_*: K_{*+1}(Q^*(Y, A)^\G) \to KK_{*}^\G(C_0(Y), A).$$

\begin{Thm}\label{thm:groupoid Roe algebra}
Assume that $Y$ is a proper $\G$-space which is $\G$-compact. 
\begin{itemize}
    \item[(1)] The reduced crossed product $A \rtimes_r \G$ is Morita equivalent to $C^*(Y, A)^{\G}$.
    \item[(2)] The following diagram strictly commutes:
    $$
    \begin{tikzcd}
    K_{*+1}(Q^*(Y, A)^\G) \arrow[r, "\partial_*"] \arrow[d, "\mathcal{P}_*"'] & K_{*}(C^*(Y, A)^\G) \arrow[d, "\mathcal{M}_{*}"] \\
    KK_{*}^\G(C_0(Y), A) \arrow[r, "\mu^{\G}_A"] & K_{*}(A \rtimes_r \G)
    \end{tikzcd}
    $$
    where $\mathcal{M}_{*}$ is the Morita equivalence isomorphism, and $\mu^\G_A$ is the Baum--Connes assembly map with coefficients in $A$.
\end{itemize}
\end{Thm}

The proof follows from \cite{BenameurRoy2022} for the canonical module $\mathcal{E}_Y$ in \eqref{eq: admissible rep}.

\subsection{Localization algebra and $KK^\G$-theory}

Inspired by Yu's localization algebra, we provide a localization algebra description of $KK^\G(Y,A)$ in this section. 

\begin{Def}
The \emph{localization algebra} $C^*_L(Y, A)^\G$ is defined as the $C^*$-algebra generated by all bounded, uniformly norm-continuous functions $f: [0, \infty) \to \IC[Y, A]^\G$ such that $\prop(f(t)) \to 0$ as $t \to \infty$.
\end{Def}

We define the evaluation map $ev: C^*_{L}(Y, A)^\G \to C^*(Y, A)^\G$ by $e(f) = f(0)$. The induced map on $K$-theory, $ev_*: K_*(C^*_{L}(Y, A)^\G) \to K_*(C^*(Y, A)^\G)$, serves as the index map in the localization setting. We now prove that this is equivalent to the Paschke dual picture.

\begin{Thm}\label{thm:localization algebra}
There is a canonical isomorphism $K_*(C^*_{L}(Y, A)^\G) \cong K_{*+1}(Q^*(Y, A)^\G)$. Under this isomorphism, the evaluation map $ev_*$ coincides with the boundary map $\partial_*$ in the Paschke dual picture. 
\end{Thm}

We shall prove \Cref{thm:localization algebra} by adapting the Eilenberg swindle argument developed by Y. Qiao and J. Roe \cite{QR2010} and the conceptual framework of M. Dadarlat, R. Willett, and J. Wu \cite{DWW2018} to the groupoid setting.

For any $C^*$-algebra $B$, let $\mathfrak{T}B$ denote the $C^*$-algebra of bounded, uniformly continuous functions $f: [0,\infty) \to B$. Evaluation at zero provides a surjective *-homomorphism $ev_0: \mathfrak{T}B \to B$. There is a canonical injective *-homomorphism $\kappa: B \to \mathfrak{T}B$ mapping an element to the constant function, which serves as a right inverse to $ev_0$. We denote by $\mathfrak{T}_0 B$ the closed ideal of functions $f \in \mathfrak{T}B$ satisfying $f(0) = 0$.

We define the \emph{localized equivariant dual Roe algebra}, denoted by $D_L^*(Y,A)^{\mathcal{G}}$, as the $C^*$-subalgebra of $\mathfrak{T}D^*(Y,A)^{\mathcal{G}}$ generated by uniformly continuous bounded functions $f: [0,\infty) \rightarrow D^*(Y,A)^{\mathcal{G}}$ such that each $f(t)$ has finite propagation and $\mathrm{prop}(f(t)) \to 0$ as $t \to \infty$. It is clear from the definitions that the localization algebra $C_L^*(Y,A)^{\mathcal{G}}$ forms a closed two-sided ideal in $D_L^*(Y,A)^{\mathcal{G}}$.

\begin{Lem}\label{lem:localized_dual_roe}
The natural inclusion maps induce a canonical $C^*$-algebra isomorphism on the quotients
$$D_L^*(Y,A)^{\mathcal{G}} / C_L^*(Y,A)^{\mathcal{G}} \to \mathfrak{T}D^*(Y,A)^{\mathcal{G}} / \mathfrak{T}C^*(Y,A)^{\mathcal{G}}. $$
\end{Lem}

\begin{proof}
First, we shall construct an explicit sequence of continuous, equivariant truncation operators $\Phi_n: \mathcal{L}_A(\mathcal{E}_A) \to \mathcal{L}_A(\mathcal{E}_A)$ to bound propagation.

Since $Y$ is a proper metric $\mathcal{G}$-space, it admits a continuous cutoff function $c: Y \to [0,1]$ i.e., $c$ satisfies that $\rho|_{\supp(c)}: \mathrm{supp}(c)\to \mathcal{G}^{(0)}$ is proper and for any $x \in \mathcal{G}^{(0)}$ and $y \in Y_x$, we have:
\[ \sum_{g \in s^{-1}(x)} c(y \cdot g^{-1}) = 1. \]
Since $\mathcal{G}$ is an \'{e}tale groupoid and $Y$ is a proper $\G$-space, the sum is a well-defined finite sum.

For a fixed integer $n \ge 1$, we can choose a uniformly locally finite open cover $\{W_i\}_{i \in I}$ of $\mathrm{supp}(c)$ such that the fiberwise diameter of each $W_i$ satisfies $\mathrm{diam}_x(W_i \cap Y_x) < 1/n$ for all $x \in \mathcal{G}^{(0)}$. Let $\{\eta_i\}_{i \in I}$ be a continuous partition of unity subordinate to this cover on $\mathrm{supp}(c)$, so $\sum_{i \in I} \eta_i(y) = 1$ for all $y \in \mathrm{supp}(c)$. We define the functions $h_i(y) = \sqrt{c(y)\eta_i(y)}$. By construction, $\mathrm{supp}(h_i) \subseteq W_i$, and $\sum_{i \in I} h_i(y)^2 = c(y)$. 

Following \Cref{def:admissible_rep}, we identify $h_i \in C_0(Y)$ with its representation on $\mathcal{E}_A$. For any bounded adjointable operator $S \in \mathcal{L}_A(\mathcal{E}_A)$, we view $S$ as a family of operators $\{S_x\}_{x \in \mathcal{G}^{(0)}}$ acting on the fibers $\mathcal{E}_x$. We define the completely positive map $\Phi_n(S)$ fiberwisely to be
\[ \Phi_n(S)_x = \sum_{g \in s^{-1}(x)} \sum_{i \in I} V_g^* \cdot (h_i|_{r(g)}) \cdot S_{r(g)} \cdot (h_i|_{r(g)}) \cdot V_g, \]
where $V_g: \mathcal{E}_{s(g)} \to \mathcal{E}_{r(g)}$ is given by the groupoid action. Since $\{\supp(g^{-1}\cdot h_i)\}$ are locally finite and $\sum_{g,i} (g^{-1}\cdot h_i)^2 = 1$, the sum is locally finite on compact subsets and converges strongly to a bounded adjointable operator of norm at most $\Vert{}S\Vert{}$. The continuity of this operator is guaranteed since $V$ is continuous in $g\in \G$ and $S$ is continuous in $x\in\G^{(0)}$.

If $T$ is $\mathcal{G}$-equivariant, we have $T_{r(g)} = V_g T_{s(g)} V_g^*$. We then have that
\[ \Phi_n(T)_x = \sum_{g \in s^{-1}(x)} \sum_{i \in I} (g^{-1} \cdot h_i|_{r(g)}) \, T_x \, (g^{-1} \cdot h_i|_{r(g)}). \]
where $g^{-1}\cdot h_i|_{r(g)}\in C_0(Y_x)$ is the pushout function of $h_i\in C_0(Y_{r(g)})$ by $g$. Thus, $\Phi_n(T)$ is $\mathcal{G}$-equivariant. Moreover, the right $\mathcal{G}$-action acts by isometries on the fibers, the fiberwise diameter of the shifted support $\supp(g^{-1} \cdot h_i|_{r(g)}) = \supp(h_i|_{r(g)}) \cdot g$ is bounded by $1/n$ by definition. Therefore, the propagation of $\Phi_n(T)$ is bounded by $1/n$.

If $T \in D^*(Y, A)^{\mathcal{G}}$ is a pseudo-local operator, the error term is given by
\[\Phi_n(T)_x - T_x = \sum_{g \in s^{-1}(x)} \sum_{i \in I} V_g^* h_i [T_{r(g)}, h_i] V_g.\]
Since $T$ is pseudo-local, each commutator $[T_{r(g)}, h_i]$ is a compact operator in $\mathcal{K}_A(\mathcal{E}_{r(g)})$. After multiplying by any compactly supported function in $C_c(Y)$, properness of the groupoid action and local finiteness of the cover ensure that only finitely many terms in the sum are non-zero. Hence $\Phi_n(T) - T$ is locally compact, and $\Phi_n(T) - T \in C^*(Y, A)^{\mathcal{G}}$.

For a function $f \in \mathfrak{T}D^*(Y, A)^{\mathcal{G}}$ and a parameter $t \in [0, \infty)$, we write $t = n + s$ where $n \in \mathbb{N} \cup \{0\}$ and $s \in [0, 1)$, and define the continuous interpolation:
\[
\Phi(f)(t) = (1 - s)\Phi_n(f(t)) + s\Phi_{n+1}(f(t)).
\]
One easily verifies that $\mathrm{prop}(\Phi(f)(t)) \to 0$ as $t \to \infty$, so $\Phi(f) \in D^*_L(Y, A)^{\mathcal{G}}$. Furthermore, $\Phi(f)(t) - f(t) \in C^*(Y, A)^{\mathcal{G}}$ uniformly for all $t \ge 0$, so $\Phi$ induces a well-defined right inverse on the quotients. If $f\in D^*_L(Y, A)^{\mathcal{G}}$ satisfies $\prop(f(t))\to 0$, then $\Phi(f)(t) - f(t)$ is locally compact at every $t \ge 0$ and satisfies $\mathrm{prop}(\Phi(f)(t) - f(t)) \le \max\{\mathrm{prop}(\Phi(f)(t)), \mathrm{prop}(f(t))\} \to 0$ as $t \to \infty$, whence $\Phi(f) - f \in C^*_L(Y, A)^{\mathcal{G}}$. The same assertion holds for all $f \in D^*_L(Y, A)^{\mathcal{G}}$ by norm approximation. Since $\Phi(\mathfrak{T}C^*(Y, A)^{\mathcal{G}}) \subseteq C^*_L(Y, A)^{\mathcal{G}}$, this implies $D^*_L(Y, A)^{\mathcal{G}} \cap \mathfrak{T}C^*(Y, A)^{\mathcal{G}} = C^*_L(Y, A)^{\mathcal{G}}$, hence that $\Phi$ induces a two-sided inverse on the quotients.
\end{proof}

The isomorphism in Lemma \ref{lem:localized_dual_roe} gives rise to a short exact sequence of $C^*$-algebras:
\[ 0 \to C_L^*(Y,A)^{\mathcal{G}} \to D_L^*(Y,A)^{\mathcal{G}} \to \mathfrak{T}D^*(Y,A)^{\mathcal{G}} / \mathfrak{T}C^*(Y,A)^{\mathcal{G}} \to 0. \]
Applying the $K$-theory functor, we obtain a boundary map:
\[ \partial_L: K_{*+1}\left(\mathfrak{T}D^*(Y,A)^{\mathcal{G}} / \mathfrak{T}C^*(Y,A)^{\mathcal{G}}\right) \to K_{*}(C_L^*(Y,A)^{\mathcal{G}}). \]

\begin{Lem}\label{lem:vanishing_DL}
The $K$-theory group $K_{*}(D_L^*(Y,A)^{\mathcal{G}})$ is trivial. Consequently, the $K$-theory boundary map $\partial_L$ defined above is an isomorphism.
\end{Lem}

\begin{proof}
Since $\E_A$ is chosen to be very ample, i.e., $\mathcal{E}_A \cong \mathcal{E}_A^\infty = \bigoplus_{i=1}^{\infty} \mathcal{E}_A$, the top-left corner inclusion of $D_L^*(Y,A)^{\mathcal{G}}$ into $D_L^*(Y,A, \mathcal{E}_A^\infty)^{\mathcal{G}}$ induces an isomorphism on $K$-theory. We define two $*$-homomorphisms $\alpha_1, \alpha_2: D_L^*(Y,A)^{\mathcal{G}} \to D_L^*(Y,A, \mathcal{E}_A^\infty)^{\mathcal{G}}$ by
\begin{align*}
\alpha_1(f)(t) &= \mathrm{diag}(f(t), 0, 0, \dots), \\
\alpha_2(f)(t) &= \mathrm{diag}(0, f(t+1), f(t+2), \dots).
\end{align*}
To see that $\alpha_2(f)$ is well-defined, notice that since $\mathrm{prop}(f(t)) \to 0$, for any two functions $\phi, \psi \in C_0(Y)$ with disjoint compact supports, there exists an integer $N$ such that
$$\prop(f(t+n))<d_Y(\supp(\phi),\supp(\psi)),$$
for all $n > N$ and $t\in\IR_+$.
Thus, the operator $\phi f(t+n) \psi = 0$ for all $n > N$. This means the infinite block diagonal matrix $\phi \alpha_2(f)(t) \psi$ has only finitely many non-zero entries, thus their direct sum is compact in $\mathcal{K}_A(\mathcal{E}_A^\infty)$.

Since $f$ is uniformly continuous, the parameter shift $r \mapsto f(t+r)$ provides a homotopy between $\alpha_2$ and the map $\alpha_3(f)(t) = \mathrm{diag}(0, f(t), f(t+1), \dots)$. By conjugating with the canonical forward shift isometry on $\ell^2(\mathbb{N})$, $\alpha_3$ is unitarily equivalent to $\alpha_4(f)(t) = \mathrm{diag}(f(t), f(t+1), \dots) = \alpha_1(f) + \alpha_2(f)$. Passing to $K$-theory, this yields:
\[ (\alpha_1)_* + (\alpha_2)_* = (\alpha_4)_* = (\alpha_3)_* = (\alpha_2)_*. \]
This means $(\alpha_1)_* = 0$. Since $(\alpha_1)_*$ is an isomorphism, we conclude that $K_{*}(D_L^*(Y,A)^{\mathcal{G}}) = 0$. 
\end{proof}

\begin{Lem}\label{lem:evaluation_iso}
The evaluation map $ev_0$ induces a $K$-theory isomorphism whose inverse is precisely induced by the inclusion of constant functions $\kappa$:
\[ \kappa_*: K_{*+1}(\mathcal{Q}^*(Y,A)^{\mathcal{G}}) \to K_{*+1}\left(\mathfrak{T}D^*(Y,A)^{\mathcal{G}} / \mathfrak{T}C^*(Y,A)^{\mathcal{G}}\right). \]
\end{Lem}
\begin{proof}
it is straightforward to check that the evaluation map $ev_0$ serves as a left inverse to $\kappa$. To prove $\kappa_*$ is an isomorphism, it suffices to prove that the kernel of $ev_0$ at the quotient level has trivial $K$-theory, i.e.,
$$K_*\Big(\mathfrak{T}_0 D^*(Y,A)^{\mathcal{G}} / \mathfrak{T}_0 C^*(Y,A)^{\mathcal{G}}\Big)=0.$$
We apply another Eilenberg swindle shifting the parameter to the left. Let $\beta_1, \beta_2: \mathfrak{T}_0 D^*(Y,A)^{\mathcal{G}} \to \mathfrak{T}_0 D^*(Y,A, \mathcal{E}_A^\infty)^{\mathcal{G}}$ be defined by
\begin{align*}
\beta_1(f)(t) &= \mathrm{diag}(f(t), 0, 0, \dots), \\
\beta_2(f)(t) &= \mathrm{diag}(0, f(t-1), f(t-2), \dots),
\end{align*}
where we set $f(s) = 0$ for $s \le 0$. For any $t \ge 0$, $f(t-N)=0$ if $N\geq \lfloor t \rfloor + 1$. Consequently, the infinite block diagonal matrix $\beta_2(f)(t)$ contains only finitely many non-zero entries, trivially ensuring pseudo-locality. An identical geometric homotopy argument to Lemma \ref{lem:vanishing_DL} shows that $K_{*}(\mathfrak{T}_0D^*(Y,A)^{\mathcal{G}}) = 0$. The identical vanishing holds for $\mathfrak{T}_0C^*(Y,A)^{\mathcal{G}}$. By the six-term exact sequence, the $K$-theory of the quotient is also trivial.
\end{proof}

We are now ready to finish the proof of \Cref{thm:localization algebra}.

\begin{proof}[Proof of \Cref{thm:localization algebra}]
By composing the $K$-theory isomorphisms established in Lemma \ref{lem:vanishing_DL} and Lemma \ref{lem:evaluation_iso}, we obtain a canonical isomorphism for the localization algebra:
\[ \partial_L \circ \kappa_*: K_{*+1}(\mathcal{Q}^*(Y,A)^{\mathcal{G}}) \xrightarrow{\cong} K_{*}\left(C_L^*(Y,A)^{\mathcal{G}}\right). \]

Consider the following commuting diagram
\[
\begin{tikzcd}[column sep=large, row sep=large]
0 \arrow[r] & C_L^*(Y,A)^{\mathcal{G}} \arrow[r] \arrow[d, "ev"] & D_L^*(Y,A)^{\mathcal{G}} \arrow[r] \arrow[d, "ev"] & \mathfrak{T}D^*(Y,A)^{\mathcal{G}} / \mathfrak{T}C^*(Y,A)^{\mathcal{G}} \arrow[r] \arrow[d, "ev"] & 0 \\
0 \arrow[r] & C^*(Y,A)^{\mathcal{G}} \arrow[r] & D^*(Y,A)^{\mathcal{G}} \arrow[r] & \mathcal{Q}^*(Y,A)^{\mathcal{G}} \arrow[r] & 0
\end{tikzcd}
\]
All the vertical maps are given by the evaluation-at-zero map.
Applying the $K$-theory functor to the above diagram, we have that
\[
\begin{tikzcd}
{K_{*+1}\left(\frac{\mathfrak{T}D^*(Y,A)^{\mathcal{G}}}{\mathfrak{T}C^*(Y,A)^{\mathcal{G}}}\right)} \arrow[r,"\partial_L", "\cong"'] \arrow[d, "ev_*", shift left] & {K_{*}\left(C_L^*(Y,A)^{\mathcal{G}}\right)} \arrow[d, "ev_*"] \\
{K_{*+1}(\mathcal{Q}^*(Y,A)^{\mathcal{G}})} \arrow[r, "\partial_*"] \arrow[u, "\kappa_*", shift left]                                                                                 & {K_{*}\left(C^*(Y,A)^{\mathcal{G}}\right)}                    
\end{tikzcd}
\]
This confirms that under the isomorphism $\partial_L \circ \kappa_*$, the evaluation map $ev_*$is equivalent to $\partial_*$. 
\end{proof}

\subsection{Classifying space for $G(X,\Gamma)$}\label{sec: equivalence of assembly maps}

In this subsection, we construct a concrete model for the classifying space of the groupoid $G(X,\Gamma)$ using the Rips complex.

The equivariant coarse Baum--Connes conjecture states that the coarse assembly map
\[ \mu_{X,\Gamma}: \lim_{d\to\infty}K^{\Gamma}_*(P_d(X))\to K_*(C^*(X)^{\Gamma}) \]
is an isomorphism. On the other hand, the groupoid Baum--Connes assembly map for $G(X,\Gamma)$ with coefficients in $\ell^\infty(X, \mathcal{K})^\Gamma$ is given by:
\[ \mu_{G(X,\Gamma)}: K_*^{\mathrm{top}}(G(X,\Gamma), \ell^\infty(X,\K)^\Gamma) \to K_*\big( \ell^\infty(X,\K)^\Gamma \rtimes_r G(X,\Gamma) \big). \]
By \Cref{cor: equivariant Roe and groupoid}, the equivariant Roe algebra $C^*(X)^\Gamma$ is canonically isomorphic to the reduced groupoid crossed product $\ell^\infty(X,\K)^\Gamma \rtimes_r G(X,\Gamma)$. It therefore suffices to prove that the topological domain $\lim_{d\to\infty}K^{\Gamma}_*(P_d(X))$ is naturally isomorphic to $K_*^{\mathrm{top}}(G(X,\Gamma), \ell^{\infty}(X,\K)^{\Gamma})$, and that this isomorphism intertwines with the respective assembly maps. 

Let $i_d: X \to P_d(X)$ be the canonical $\Gamma$-equivariant inclusion. For any $R \ge 0$, define
\[ W_{R,d} = \{(x,w) \in X \times P_d(X) \mid d_{P_d(X)}(i_d(x), w) \le R\}. \]
Since $i_d$ is $\Gamma$-equivariant, $W_{R,d}$ is $\Ga$-invariant under the canonical $\Ga$-diagonal action.
For any metric space $M$, we shall denote by $\beta_u M$ the Gelfand spectrum of the $C^*$-algebra $C_{ub}(M)$ of all bounded and uniformly continuous functions on $M$, which is called the \emph{uniform compactification} (or \emph{Samuel compactification}) of $M$, see \cite{Samuel1948}. In particular, if two nets $\{x_i\}$ and $\{y_i\}$ in $M$ indexed by the same directed set converge to $z_1$ and $z_2$ in $\beta_u M$ respectively and satisfy $\lim_i d_M(x_i, y_i) = 0$, then $|f(x_i) - f(y_i)| \to 0$ for all $f \in C_{ub}(M)$, which implies $z_1 = z_2$.

\begin{Lem}\label{lem:Xd_construction}
The uniform boundary space
\[ X_d = \bigcup_{R \ge 0} \overline{W_{R,d}/\Gamma} \subseteq \beta_u\big((X \times P_d(X))/\Gamma\big) \]
is a locally compact Hausdorff space equipped with a continuous anchor map $\rho: X_d \to \beta(X/\Gamma)$, extended by the map $(X \times P_d(X))/\Gamma\to X/\Ga$, defined by $[x, w]\mapsto [x]$.
\end{Lem}

\begin{proof}
Because the $\Gamma$-action on $X$ is free and proper, the diagonal right action on $X \times P_d(X)$ is also proper and free. Consequently, the quotient space $Z_d = (X \times P_d(X))/\Gamma$ is a proper metric space, equipped with the quotient metric. By taking the uniform compactification of the locally compact Hausdorff quotient $Z_d$, we have that $\beta_u(Z_d)$ is a compact Hausdorff space.

Since $W_{R,d}$ is $\Gamma$-invariant, its quotient $W_{R,d}/\Gamma$ is a closed subset of $Z_d$, so its closure $\overline{W_{R,d}/\Gamma}$ in $\beta_u(Z_d)$ is compact Hausdorff. Moreover, since the function $[x, w] \mapsto \min\{d_{P_d(X)}(i_d(x), w), R+2\}$ belongs to $C_{ub}(Z_d)$, its continuous extension to $\beta_u(Z_d)$ shows that $\overline{W_{R,d}/\Gamma}$ is contained in the interior of $\overline{W_{R+1,d}/\Gamma}$ for every $R \ge 0$. Therefore, $X_d = \bigcup_{R \ge 0} \overline{W_{R,d}/\Gamma}$ is an open subset of $\beta_u(Z_d)$ and hence a locally compact Hausdorff space.

The projection map $p_1: X \times P_d(X) \to X$ is uniformly continuous and $\Gamma$-equivariant. It descends to a uniformly continuous map $\widetilde{p}_1: Z_d \to X/\Gamma$. By the universal property of the uniform compactification, $\widetilde{p}_1$ extends uniquely to a continuous map on the compactifications $\overline{p}_1: \beta_u(Z_d) \to \beta_u(X/\Gamma)$. Because $X$ is uniformly discrete with bounded geometry, any bounded function on $X/\Gamma$ is uniformly continuous, making $\beta_u(X/\Gamma)$ naturally identified with $\beta(X/\Gamma)$. The restriction of $\overline{p}_1$ to $X_d$ provides the required continuous anchor map $\rho: X_d \to \beta(X/\Gamma)$.
\end{proof}

\begin{Lem}\label{lem:Xd_metric_cocompact}
The space $X_d$ is a proper metric $G(X,\Gamma)$-space (in the sense of \Cref{def:proper G-space}). Furthermore, the left action of $G(X,\Gamma)$ on $X_d$ is fiberwise isometric and cocompact.
\end{Lem}

\begin{proof}
We first show the proper metric $G(X,\Gamma)$-space structure. On $Z_d = (X \times P_d(X))/\Gamma$, the anchor map to $X/\Gamma$ is given by $[x, y] \mapsto [x]$. For a fixed orbit $[x] \in X/\Gamma$, the fiber is naturally identified with $P_d(X)$ via $y \mapsto [x,y]$. We define the metric between two points in the $[x]$-fiber by
\[ d_{[x]}\big([x, y_1], [x, y_2]\big) := d_{P_d(X)}(y_1, y_2). \]

To construct the metric concretely on each fiber of the uniform boundary space $X_d$, consider any two elements $z_1, z_2 \in X_d$ sharing the same anchor $\rho(z_1) = \rho(z_2) = \omega \in \beta(X/\Gamma)$. Because $X_d = \bigcup_{R \ge 0} \overline{W_{R,d}/\Gamma}$, there exists $R \ge 0$ such that both $z_1, z_2 \in \overline{W_{R,d}/\Gamma}$. For any net of pairs $(\alpha_i, \beta_i) \in W_{R,d}/\Gamma \times_{X/\Gamma} W_{R,d}/\Gamma$ converging to $(z_1, z_2)$, one directly has $d_{\rho(\alpha_i)}(\alpha_i, \beta_i) \le 2R$ by the triangle inequality, and the function $(\alpha, \beta) \mapsto d_{\rho(\alpha)}(\alpha, \beta)$ is uniformly continuous on $W_{R,d}/\Gamma \times_{X/\Gamma} W_{R,d}/\Gamma$. By the universal property of the uniform compactification (and a similar argument to \Cref{lem:key lemma}, since $X/\Gamma$ is discrete), this function extends uniquely to a continuous function on $\overline{W_{R,d}/\Gamma} \times_{\beta(X/\Gamma)} \overline{W_{R,d}/\Gamma}$, so the limit
\[ d_\omega(z_1, z_2) = \lim_{i} d_{\rho(\alpha_i)}\big(\alpha_i, \beta_i\big) \]
exists and is independent of the choice of the approximating net. Moreover, if $d_\omega(z_1, z_2) = 0$, then $\lim_i d_{\rho(\alpha_i)}(\alpha_i, \beta_i) = 0$, which implies $z_1 = z_2$. This produces a well-defined, continuous family of metrics on the fibers of $X_d$.

Recall that $G(X,\Gamma) = \bigcup_{S \ge 0} \overline{\Delta_S/\Gamma} \subseteq \beta((X \times X)/\Gamma)$. For $g = [x_1, x_2] \in \Delta_S/\Gamma$ and $z = [x_2, y] \in W_{R,d}/\Gamma$, we set
\[ [x_1, x_2] \cdot [x_2, y] = [x_1, y]. \]
By the triangle inequality in $P_d(X)$, $d_{P_d(X)}(i_d(x_1), y) \le S + R$, meaning the resulting element belongs to $W_{S+R,d}/\Gamma$. This composition then extends continuously to a proper left action of $G(X,\Gamma)$ on $X_d$ by using a similar argument to \Cref{lem:key lemma}.

To see this action is fiberwise isometric, we first check this on the fibers over $X/\Gamma$. Applying $g = [x_1, x_2]$ on $[x_2, y_1]$ and $[x_2, y_2]$, we get $[x_1,y_1]$ and $[x_1,y_2]$ respectively. Then
\[d_{[x_1]}([x_1,y_1], [x_1,y_2])=d_{[x_2]}([x_2,y_1], [x_2,y_2])=d_{P_d(X)}(y_1,y_2).\]
By continuity, its continuous extension is still fiberwise isometric. To verify the properness of the metric, let $K \subseteq X_d$ be a compact subset; then $K \subseteq \overline{W_{R_0,d}/\Gamma}$ for some $R_0 \ge 0$. By the triangle inequality, for any $R \ge 0$, the closed fiberwise $R$-neighborhood of $K$ is contained in $\overline{W_{R_0+R,d}/\Gamma}$, which is compact.

Finally, we shall prove that the action is cocompact. The inclusion $i_d: X\to P_d(X)$ is an equivariant coarse equivalence, thus there exists $D > 0$ such that the vertex set $i_d(X)$ is a $D$-net in $P_d(X)$. 
For any element $z = [x_1, y] \in W_{R,d}/\Gamma \subseteq Z_d$,  we can find a vertex $x_2 \in X$ such that $d_{P_d(X)}(i_d(x_2), y) \le D$. This means the element $[x_2, y]$ strictly belongs to $W_{D,d}/\Gamma$. Since $[x_1, y]\in W_{R,d}$ and $[x_2, y]\in W_{D,d}$, we conclude that $d(x_1,x_2)\leq R+D$. Consider the groupoid element $g = [x_1, x_2]\in \Delta_{R+D}/\Ga\subseteq G(X,\Gamma)$. Since $g \cdot [x_2, y] = [x_1, y] = z$, we have that the $G(X,\Gamma)$-orbit of $z$ in $W_{R,d}/\Gamma$ intersects $W_{D,d}/\Gamma$, for any $R\geq 0$. 

By continuity, any element $w \in \overline{W_{R,d}/\Gamma} \subseteq X_d$ is the limit of a net $z_i \in W_{R,d}/\Gamma$. Each $z_i$ can be factored as $z_i = g_i \cdot v_i$ for $v_i \in W_{D,d}/\Gamma$ and $g_i \in \Delta_{D+R}/\Gamma$. Because $\overline{W_{D,d}/\Gamma}$ and $\overline{\Delta_{R+D}/\Gamma}$ are compact, we can pass to converging subnets $v_i \to v \in \overline{W_{D,d}/\Gamma}$ and $g_i \to g \in \overline{\Delta_{R+D}/\Gamma}$. By the continuity of the action, $w = g \cdot v$. This proves that
\[ X_d = G(X,\Gamma) \cdot \overline{W_D/\Gamma}. \]
This shows that the action is cocompact.
\end{proof}

Recall that in \cite[Appendix B]{SpaWil2017}, J.~\v{S}pakula and R.~Willett showed the limit space construction for each fiber of the coarse groupoid $G(X)$. In our setting, since $\Ga$ acts freely and isometrically on $X$, the fibers of $G(X)$ that are glued together by the $\Ga$-action necessarily have isometric limit spaces. Thus, we similarly obtain a well-defined limit space construction for $G(X,\Ga)$, which identifies each fiber $r^{-1}(\omega)$ over $\omega \in \beta(X/\Ga)$ with a limit space $X(\omega')$, where $\omega'\in\beta X$ with $\pi(\omega')=\omega\in\beta(X/\Ga)$. For simplicity, we denote $X(\omega')$ by $X(\omega)$. Under this identification, the fiber $(X_d)_\omega$ is precisely the Rips complex $P_d(X(\omega))$ of the limit space $X(\omega)$. Since $X$ has bounded geometry, each limit space $X(\omega)$ also has bounded geometry.

For each $R \ge 0$, consider the $\Ga$-invariant subspace
\[ \wt{W}_{R,d} := \big\{ (x, w) \in X \times P_d(X) \;\big|\; \supp(w) \subseteq B_X(x, R) \big\}. \]
Since the inclusion $i_d: X \to P_d(X)$ is an equivariant coarse equivalence, we conclude that $X_d = \bigcup_{R \ge 0} \overline{\wt{W}_{R,d}/\Ga} \subseteq \beta_u((X \times P_d(X))/\Ga)$. 

\begin{Lem}\label{lem:polyhedral_rips_charts}
For any fixed $d \ge 0$, the space $X_d$ satisfies the following properties.
\begin{enumerate}
    \item[(1)] For each $R \ge 0$, the compact subspace $\overline{\wt{W}_{R,d}/\Ga} \subseteq X_d$ decomposes into a finite disjoint union $\bigsqcup_{j=1}^N (\overline{T_j} \times K_j)$, where each $\overline{T_j} \subseteq \beta(X/\Ga)$ is clopen and each $K_j$ is a finite complex of dimension at most $N_d$.
    \item[(2)] For any closed $G(X,\Ga)$-invariant subcomplex $Z \subseteq X_d$ that is locally a product $\overline{T_j} \times L_j$, there exist an open $G(X,\Ga)$-invariant neighborhood $U_Z$ of $Z$ and a continuous, $G(X,\Ga)$-equivariant map $p_Z: U_Z \to Z$ such that $p_Z(z) = z$ for all $z \in Z$ and $\lim_{d_{\rho(y)}(y, Z) \to 0} d_{\rho(y)}(y, p_Z(y)) = 0$ uniformly on fibers.
\end{enumerate}
\end{Lem}

\begin{proof}
(1) Over each orbit $[x] \in X/\Ga$, the fiber of $\wt{W}_{R,d}/\Ga$ is the finite subcomplex $P_d(B_X(x, R))$. By \Cref{lem:partition}, we can partition $\Delta_R/\Ga$ into finitely many disjoint subsets $V_1, \dots, V_m$ such that $r_0|_{V_k}: V_k \to X/\Ga$ is injective for each $1 \le k \le m$. For any $[x] \in X/\Ga$, since each $V_k$ contains at most one element with range $[x]$, this partition assigns a unique label in $\{1, \dots, m\}$ to each vertex of $B_X(x, R)$. This coloring allows us embed the fiber $P_d(B_X(x, R))$ as a subcomplex $K([x])$ of the fixed $(m-1)$-simplex $\Delta^{m-1}$. Since $\Delta^{m-1}$ has only finitely many subcomplexes, say $K_1, \dots, K_N$, we can partition $X/\Ga$ into finitely many discrete subsets
\[ T_j := \big\{ [x] \in X/\Ga \;\big|\; K([x]) = K_j \big\}, \qquad j = 1, \dots, N. \]
Over each $T_j$, every point $[x] \in T_j$ has the exact same finite complex $K_j$ as its fiber, which trivializes the restriction of $\wt{W}_{R,d}/\Ga$ over $T_j$ as the product space $T_j \times K_j$. Since $\beta_u(T_j \times K_j) \cong \overline{T_j} \times K_j$, taking the closure in the uniform compactification gives $\overline{\wt{W}_{R,d}/\Ga}$ a local trivialization as a product of a clopen subset of $\beta(X/\Ga)$ with a finite complex. Passing to the union over all $R \ge 0$ gives the desired local product charts for $X_d$.

(2) Let $Z \subseteq X_d$ be a closed $G(X,\Ga)$-invariant subcomplex. Passing to a barycentric subdivision if necessary, we may assume that $Z$ is a full subcomplex in each fiber (i.e., any simplex whose vertices lie in $Z$ is itself contained in $Z$). Writing $y \in X_d$ as $y = \sum_v a_v v$, we define a continuous $G(X,\Ga)$-invariant function $b_Z: X_d \to [0,1]$ by $b_Z(y) := \sum_{v \in Z} a_v$. Since $Z$ is a full subcomplex, we have $y \in Z$ if and only if $b_Z(y) = 1$. On the open $G(X,\Ga)$-invariant neighborhood $U_Z := \{y \in X_d \mid b_Z(y) > 1/2\}$ of $Z$, we define $p_Z: U_Z \to Z$ by normalizing the coordinates supported on $Z$:
\begin{equation}\label{eq:polyhedral_retraction}
p_Z(y) := \frac{\sum_{v \in Z} a_v v}{b_Z(y)}.
\end{equation}
Because the $G(X,\Ga)$-action permutes the vertices isometrically and preserves $Z$, it is direct to check that $p_Z$ is continuous, $G(X,\Ga)$-equivariant, and satisfies $p_Z(z) = z$ for all $z \in Z$. Moreover, for any $y \in U_Z$, $p_Z(y)$ lies in the same simplex as $y$, and the fiberwise distance $d_{\rho(y)}(y, p_Z(y))$ tends uniformly to zero as $b_Z(y) \to 1$, equivalently, as $d_{\rho(y)}(y, Z) \to 0$.
\end{proof}

\begin{Lem}\label{lem:K_homology_iso}
For any $d \ge 0$, there is a canonical isomorphism:
\[ K_*^\Gamma(P_d(X)) \cong K_*\big(C^*_L(X_d, \ell^\infty(X,\K)^\Gamma)^{G(X,\Gamma)}\big). \]
\end{Lem}

\begin{proof}
On the dense subset $X/\Gamma$, each fiber $(X_d)_{[x]}$ is identified with $P_d(X)$ via $\phi_x: y \mapsto [x, y]$. Let $\mu_0$ be the standard Lebesgue measure on each simplex of $P_d(X)$. For any $f \in C_c(X_d)$, we have $\supp(f) \subseteq \overline{W_{R,d}/\Gamma}$ for some $R > 0$. Since $X$ has bounded geometry, the $R$-balls in $P_d(X)$ contain a uniformly bounded number of simplices, so the fiberwise integral $[x] \mapsto \int_{P_d(X)} f([x,y])\, d\mu_0(y)$ is uniformly bounded on $X/\Gamma$ and hence extends uniquely to a continuous function on $\beta(X/\Gamma)$. In the local product charts $\overline{T_j} \times K_j$ of $\overline{\wt{W}_{R,d}/\Gamma}$ by \Cref{lem:polyhedral_rips_charts}, this extension corresponds to integrating on the finite complex $K_j$ under Lebesgue measure, which defines a continuous family of strictly positive Radon measures on the fibers of $X_d$.

Denote $A = \ell^\infty(X,\K)^\Gamma$ with $\K = \K(\ell^2(\Gamma) \otimes H)$, and consider the canonical admissible module
\[ \E_{X_d} = L^2(X_d, \mu) \otimes_{C(\beta(X/\Gamma))} \ell^2(G(X,\Gamma)) \otimes_{C(\beta(X/\Gamma))} A \otimes H_0 \]
with respect to the full $\rho$-system as in \Cref{rem:rho_system_existence}. Fixing a base orbit $[x_0] \in X/\Gamma$, we have $(X_d)_{[x_0]} \cong P_d(X)$, $r^{-1}([x_0]) \cong X$, and $A_{[x_0]} \cong \K$. Since $\L_{\K}(\mathcal{H}_0 \otimes \K) \cong \B(\mathcal{H}_0 \otimes \ell^2(\Gamma) \otimes H)$, the fiber module $(\E_{X_d})_{[x_0]}$ identifies canonically with the equivariant ample $P_d(X)$-module
$$\mathcal{H}_{P_d(X)} = L^2(P_d(X), \mu_0) \otimes \ell^2(X) \otimes \ell^2(\Gamma) \otimes H \otimes H_0,$$
on which the isotropy group of $G(X,\Gamma)$ at $[x_0]$ acts as $\Gamma$. Restricting any operator family $T = \{T_\omega\}_{\omega \in \beta(X/\Gamma)} \in C^*(X_d, A)^{G(X,\Gamma)}$ to the fiber over $[x_0]$ yields a $\Gamma$-equivariant operator $\Phi_0(T) = T_{[x_0]} \in C^*(P_d(X))^\Gamma$ with $\prop(\Phi_0(T)) \le \prop(T)$ and $\|\Phi_0(T)\| = \|T\|$ as $G(X,\Gamma)$ acts transitively on the dense subset $X/\Gamma$. Applying $\Phi_0$ pointwise in $t \in [0,\infty)$ embeds $C^*_L(X_d, A)^{G(X,\Gamma)}$ as a $C^*$-subalgebra via the injective $*$-homomorphism
\[ \Phi: C^*_L(X_d, A)^{G(X,\Gamma)} \hookrightarrow C^*_L(P_d(X))^\Gamma. \]

To construct the inverse, we cannot extend an operator $\widetilde{T} \in \mathbb{C}[P_d(X)]^\Gamma$ directly to $X_d$ by translating it along $X/\Gamma$. To see the reason, we first restrict to a local product chart $\overline{T_j} \times K_j$. Over $\overline{T_j} \times K_j$, operators in $C^*(X_d, A)^{G(X,\Gamma)}$ restrict to $\mathcal{K}\big(L^2(K_j, \mu_0) \otimes \ell^2(X) \otimes H_0\big) \otimes \ell^\infty(T_j, \mathcal{K})$ by local compactness. However, translating $\widetilde{T}$ along $T_j$ only gives an element in $\ell^\infty\big(T_j, \mathcal{K}(L^2(K_j, \mu_0) \otimes \ell^2(X) \otimes H_0) \otimes \mathcal{K}\big)$. Nevertheless, if the action of $\widetilde{T}$ on the factor $L^2(K_j, \mu_0) \otimes \ell^2(X) \otimes H_0$ factors through a \emph{finite-dimensional} subspace $V_j \subseteq L^2(K_j, \mu_0) \otimes \ell^2(X) \otimes H_0$ spanned by fixed finite-support sections over $\overline{T_j}$, then $\mathcal{K}(V_j)$ is finite-dimensional and $\mathcal{K}(V_j) \otimes \ell^\infty(T_j, \mathcal{K}) = \ell^\infty(T_j, \mathcal{K}(V_j) \otimes \mathcal{K})$, so that $\widetilde{T}$ lies in $\mathrm{Im}(\Phi_0)$.

We therefore compress the factor $L^2(P_d(X), \mu_0) \otimes \ell^2(X) \otimes H_0$ onto finite-dimensional subspaces on each chart via a covering isometry path for $\mathrm{id}_{P_d(X)}$. For each $n \ge 1$, choose a sufficiently fine iterated barycentric subdivision of $P_d(X)$, and let $\{C_\alpha\}$ be the interiors of its simplices of positive  Lebesgue measure, which partition $P_d(X)$ into disjoint measurable blocks with diameter $\varepsilon_n\to 0$. For each $C_\alpha$, let $V(C_\alpha) \subseteq X$ be the finite vertex set of the carrier simplex containing $C_\alpha$, and define the unit vectors
\[
\xi_{C_\alpha} := \mu_0(C_\alpha)^{-1/2}\chi_{C_\alpha} \in L^2(C_\alpha, \mu_0), \qquad v_{C_\alpha} := |V(C_\alpha)|^{-1/2}\sum_{x \in V(C_\alpha)} \delta_x \in \ell^2(X).
\]
Fix a unit vector $\eta_0 \in H_0$. Using the free $\Gamma$-action and Fell's absorption on $\ell^2(\Gamma) \otimes H$, we choose a $\Gamma$-equivariant isometry $W_n$ on $\mathcal{H}_{P_d(X)}$ supported in $\bigsqcup_\alpha (\overline{C}_\alpha \times \overline{C}_\alpha)$ that maps each $L^2(C_\alpha, \mu_0) \otimes \ell^2(X) \otimes \ell^2(\Gamma) \otimes H \otimes H_0$ onto $(\mathbb{C}\xi_{C_\alpha}) \otimes (\mathbb{C}v_{C_\alpha}) \otimes \ell^2(\Gamma) \otimes H \otimes (\mathbb{C}\eta_0)$ by absorbing $L^2(C_\alpha, \mu_0) \otimes \ell^2(X) \otimes H_0$ into the internal coefficient space $\ell^2(\Gamma) \otimes H$. Since each finite complex $K_j$ contains only finitely many cells $C_\alpha$ and each $v_{C_\alpha}$ is a fixed linear combination of the clopen bisections covering $\overline{\Delta_{R+d}/\Gamma}$, the compressed operator $W_n \widetilde{T} W_n^*$ acts on $L^2(K_j, \mu_0) \otimes \ell^2(X) \otimes H_0$ through the finite-dimensional subspace $V_j = \mathrm{span}\{\xi_{C_\alpha} \otimes v_{C_\alpha} \otimes \eta_0 \mid C_\alpha \subseteq K_j\}$ and thus belongs to $\mathrm{Im}(\Phi_0)$. Connecting $W_n$ and $W_{n+1}$ on $t \in [n, n+1]$ via rotation in $H_0$ as in \cite{Yu1997} yields a continuous covering isometry path $W(t)$ for $\mathrm{id}_{P_d(X)}$ with $\mathrm{Prop}(W(t)) \le \varepsilon_{[ t ]} \to 0$. Then $\mathrm{Ad}_{W(t)}$ maps $C^*_L(P_d(X))^\Gamma$ into $\mathrm{Im}(\Phi)$. Moreover, for any $T = \Phi(S) \in \mathrm{Im}(\Phi)$, since $W(t)$ projects onto the finite-dimensional subspace $V_j$ on each chart, the mixed products $W(t)T(t)$ and $T(t)W(t)^*$ also factor through $V_j$ on one side and hence lie in $\mathrm{Im}(\Phi)$. Consequently, the standard $2 \times 2$ rotation homotopies between $W(t)$ and the identity take place both in $M_2(C^*_L(P_d(X))^\Gamma)$ and in $M_2(\mathrm{Im}(\Phi))$, showing that $(\mathrm{Ad}_{W(t)})_*$ is the two-sided inverse to $\Phi_*$ on $K$-theory.
\end{proof}

\subsection{Functoriality of $K$-theory of localization algebras}\label{subsec:loc_to_KK}

In this subsection, we prove that for the equivariant coarse groupoid $G(X, \Ga)$ and the coefficient algebra $A = \ell^\infty(X, \K)^\Ga$, the $K$-theory of the localization algebra on $X_d$ is canonically isomorphic to the groupoid equivariant $KK$-theory. 

Recall from \Cref{thm:localization algebra} that for any proper, $G(X,\Ga)$-compact metric $G(X,\Ga)$-space $Y$, we have a canonical isomorphism
\[ \partial_L \circ \kappa_* : K_{*+1}\big(Q^*(Y, A)^{G(X,\Ga)}\big) \xrightarrow{\;\cong\;} K_*\big(C^*_L(Y, A)^{G(X,\Ga)}\big). \]
Composing its inverse with $\mathcal{P}_* : K_{*+1}(Q^*(Y, A)^{G(X,\Ga)}) \to KK^{G(X,\Ga)}_*(C_0(Y), A)$ from \cite{BenameurRoy2022}, we obtain a canonical comparison homomorphism
\begin{equation}\label{eq:eta_comparison_def}
\eta_Y := \mathcal{P}_* \circ (\partial_L \circ \kappa_*)^{-1} : K_*\big(C^*_L(Y, A)^{G(X,\Ga)}\big) \longrightarrow KK^{G(X,\Ga)}_*\big(C_0(Y), A\big).
\end{equation}
Notice that we do not assume $\mathcal{P}_*$ to be an isomorphism in general, it is proved in \cite{BenameurRoy2021} that $\mathcal{P}_*$ is an isomorphism for transformation groupoids. Our goal in this subsection is to prove directly that $\eta_{X_d}$ is an isomorphism for every $d \ge 0$ by homological properties of localization algebras with a cutting-and-pasting reduction to finite group actions.

Since the unit space $\beta(X/\Ga)$ is not second-countable and the coefficient algebra $A = \ell^\infty(X, \K)^\Ga$ is not separable, we show that the comparison can be reduced to the case of second-countable groupoids and separable coefficient subalgebras, following the approach of \cite[Section 5.1]{STY2002}.

\begin{Lem}\label{lem:countable_data_reduction}
For any $d \ge 0$, there exist a \emph{second-countable}, locally compact, Hausdorff, \'{e}tale groupoid $\G_0$ with unit space $B_0$,  and a continuous action of $\G_0$ on $\beta(X/\Ga)$ such that
\[
G(X,\Ga) \cong \beta(X/\Ga) \rtimes \G_0 \quad \text{and} \quad X_d \cong \beta(X/\Ga) \times_{B_0} Y_d,
\]
where $Y_d$ is a \emph{second-countable}, proper, cocompact metric $\G_0$-space over $B_0$ with the same local complex structure as $X_d$. Regarding $A = \ell^\infty(X,\K)^\Ga$ as a $\G_0$-$C^*$-algebra, we have canonical isomorphisms
\[
KK^{G(X,\Ga)}_*\big(C_0(X_d), A\big) \cong KK^{\G_0}_*\big(C_0(Y_d), A\big) \quad \text{and} \quad C^*_L(X_d, A)^{G(X,\Ga)} \cong C^*_L(Y_d, A)^{\G_0}.
\]
Moreover, there are canonical isomorphisms
\[
\lim_{\substack{D\subseteq A \\ D \text{ separable}}} K_*\big(C^*_L(Y_d, D)^{\G_0}\big) \xrightarrow{\;\cong\;} K_*\big(C^*_L(Y_d, A)^{\G_0}\big), \quad
\lim_{\substack{D\subseteq A \\ D \text{ separable}}} KK^{\G_0}_*\big(C_0(Y_d), D\big) \xrightarrow{\;\cong\;} KK^{\G_0}_*\big(C_0(Y_d), A\big),
\]
where the direct limits are taken over all \emph{separable} $\G_0$-invariant $C(B_0)$-subalgebras $D \subseteq A$.
\end{Lem}

\begin{proof}
We follow the method of \cite[Proposition 3.4]{STY2002}. By \Cref{lem:partition}, each $\overline{\Delta_n/\Ga}$ ($n \in \IN$) is a finite union of clopen bisections. Taking a countable family $\S$ of clopen bisections in $G(X,\Ga)$ closed under multiplication, inversion, intersections, and differences of domains that contains the partitions of $\overline{\Delta_n/\Ga}$ and the clopen sets $\overline{T_j} \subseteq \beta(X/\Ga)$ from \Cref{lem:polyhedral_rips_charts} (1) for all $R \in \IN$, let $C(B_0) \subseteq C(\beta(X/\Ga))$ be the separable $C^*$-subalgebra generated by $\{\chi_{s(V)} \mid V \in \S\}$. By \cite[Proposition 3.4]{STY2002}, its Gelfand spectrum $B_0$ is a second-countable compact Hausdorff space equipped with a canonical quotient map $p : \beta(X/\Ga) \to B_0$, and $\S$ descends to a second-countable \'{e}tale quotient groupoid $\G_0$ over $B_0$ such that $G(X,\Ga) \cong \beta(X/\Ga) \rtimes \G_0$.

Recall from \Cref{lem:polyhedral_rips_charts} (1) that for each $R \in \IN$, $\overline{\wt{W}_{R,d}/\Ga} = \bigsqcup_{j=1}^N (\overline{T_j} \times K_j)$. Since $\chi_{\overline{T_j}} \in C(B_0)$, the image $U_j = p(\overline{T_j})$ is clopen in $B_0$ and satisfies $\overline{T_j} = p^{-1}(U_j)$. Quotienting the first coordinate by $p \times \mathrm{id} : \overline{T_j} \times K_j \to U_j \times K_j$ on each piece defines a second-countable quotient space
\[ Y_{R,d} := \bigsqcup_{j=1}^N \big(U_j \times K_j\big) \]
of $\overline{\wt{W}_{R,d}/\Ga}$ over $B_0$. Because the partitions are chosen compatibly, the inclusions $\overline{\wt{W}_{R,d}/\Ga} \hookrightarrow \overline{\wt{W}_{R+1,d}/\Ga}$ commute with the quotient maps $p \times \mathrm{id}$ and induce closed embeddings $Y_{R,d} \hookrightarrow Y_{R+1,d}$. Gluing the spaces $Y_{R,d}$ together along these embeddings yields a second-countable locally compact space $Y_d := \bigcup_{R \in \IN} Y_{R,d}$ over $B_0$ satisfying $X_d \cong \beta(X/\Ga) \times_{B_0} Y_d$. Moreover, the quotient map $X_d \to Y_d$ identifies each fiber $(X_d)_\omega$ isometrically with $(Y_d)_{p(\omega)}$, through which the $G(X,\Ga)$-action on $X_d$ descends to a proper cocompact fiberwise isometric $\G_0$-action on $Y_d$.

As we have the groupoid $G(X,\Ga) \cong \beta(X/\Ga) \rtimes \G_0$, it follows by a similar method to \cite[Lemma 4.1]{STY2002} to show $KK^{G(X,\Ga)}_*\big(C_0(X_d), A\big) \cong KK^{\G_0}_*\big(C_0(Y_d), A\big)$. Since $C_0(X_d) \cong C(\beta(X/\Ga)) \ox_{C(B_0)} C_0(Y_d)$, $\E_{X_d} \cong C(\beta(X/\Ga)) \ox_{C(B_0)} \E_{Y_d}$, and all $A$-linear adjointable operators commute with the central action of $C(\beta(X/\Ga))$, pulling back equivariant finite-propagation operators along $p$ yields the second isomorphism $C^*_L(X_d, A)^{G(X,\Ga)} \cong C^*_L(Y_d, A)^{\G_0}$, we leave the details to the reader.

Finally, any countable subset of $A$ is contained in a separable $\G_0$-invariant $C(B_0)$-subalgebra $D \subseteq A$. Since a continuous path $f \in C^*_L(Y_d, A)^{\G_0}$ is completely determined by its values at rational times $t \in \IQ_+$, and each operator $f(t) \in C^*(Y_d, A)^{\G_0}$ is determined by countably many matrix entries in $A$ with respect to a countable dense subset of the geometric Hilbert $C(B_0)$-module $L^2(Y_d) \ox_{C(B_0)} \ell^2(\G_0) \ox H_0$ on the compact subcomplexes $Y_{R,d}$, the entire path $f$ belongs to $C^*_L(Y_d, D)^{\G_0}$ for some separable $D \subseteq A$. This implies that $C^*_L(Y_d, A)^{\G_0} = \lim C^*_L(Y_d, D)^{\G_0}$. The second isomorphism follows by an identical argument since $C_0(Y_d)$ is separable and every Kasparov cycle is countably generated, see \cite{LeGall1999}.
\end{proof}

Thanks to \Cref{lem:countable_data_reduction}, in the remainder of this subsection, we work over the second-countable groupoid $\mathcal{G}_0$, a fixed separable $\mathcal{G}_0$-invariant $C(B_0)$-subalgebra $D \subseteq A$, and closed $\mathcal{G}_0$-invariant subcomplexes $W, W'$ of $Y_d$. For each such space $W$, we use the canonical admissible module $\mathcal{E}_W = L^2(W) \otimes_{C(B_0)} \ell^2(\mathcal{G}_0) \otimes_{C(B_0)} D \otimes H_0$.

\begin{Pro}\label{pro:functoriality_loc}
Let $W$ and $W'$ be proper, cocompact metric $\G_0$-spaces as in \Cref{lem:polyhedral_rips_charts}, and let $f : W \to W'$ be a continuous, $\G_0$-equivariant map.
\begin{enumerate}
    \item[(1)] The map $f$ is proper and equi-uniformly continuous on fibers.
    \item[(2)] The map $f$ induces a canonical homomorphism $f_* : K_*(C^*_L(W, D)^{\G_0}) \to K_*(C^*_L(W', D)^{\G_0}).$
\end{enumerate}
\end{Pro}

\begin{proof}
(1) Since $W$ is cocompact, we can choose a compact subset $K \subseteq W$ intersecting with every $\G_0$-orbit. By the properness of the fiberwise metric, for any $R > 0$, the set of pairs $(x,y) \in K \times_{B_0} W$ with $d(x,y) \le R$ is compact. The continuity of $f$ on this compact set, combined with the fiberwise isometric $\G_0$-action, yields the uniform fiberwise modulus $\omega_f$ with $\lim_{r\to 0^+}\omega_f(r) = 0$. Moreover, if $L \subseteq W'$ is compact and $g \cdot k \in f^{-1}(L)$ for some $k \in K$, then $g \cdot f(k) \in L$. Since the $\G_0$-action on $W'$ is proper, $g$ must lie in a compact subset of $\G_0$, which implies that $f^{-1}(L)$ is compact.

(2) Fix a sequence $\delta_n \to 0$. Since $W$ is a proper $\G_0$-space, there exists a continuous cutoff function $c: W \to [0,1]$ such that $\sum_{g \in r^{-1}(x)} c(g^{-1}\cdot w) = 1$ for all $x \in B_0$ and $w \in W_x$. Since $W$ is cocompact, $\supp(c)$ is compact. By \Cref{lem:polyhedral_rips_charts} (1) and the continuity of $f$, we can cover $\supp(c)$ by finitely many open subsets $O_1, \dots, O_m$ such that each $f(O_i)$ is contained in a local product chart $O'_i \cong U_i \times K'_i \subseteq W'$ of fiberwise diameter less than $\delta_n$, where $U_i \subseteq B_0$ is clopen and $K'_i$ is an open subset of a finite complex. Let $\{\theta_i\}_{i=1}^m$ be a continuous partition of unity on $\supp(c)$ subordinate to $\{O_i\}_{i=1}^m$, and set $h_i(w) := \sqrt{c(w)\theta_i(w)} \in C_c(O_i)$. Then $\sum_{g \in r^{-1}(x)} \sum_{i=1}^m h_i(g^{-1}\cdot w)^2 = 1$ for all $w \in W_x$.

For each $1 \le i \le m$, consider the submodule $\E_{O'_i} := \overline{\pi'(C_0(O'_i))\E_{W'}} \subseteq \E_{W'}$. Over $U_i$, the representation of $C_0(O'_i) \cong C(U_i) \ox C_0(K'_i)$ on $\E_{O'_i}$ is induced by the representation of the separable commutative $C^*$-algebra $C(U_i)$ and $C_0(K'_i)$ individually. By Kasparov's theorem \cite[Theorem 5]{Kasparov1980} (see also \cite[Theorem B.6]{GWY2022}), there exists a $D$-linear, $C(U_i)$-linear isometry $w_i : \overline{\pi(C_0(O_i))\E_W} \to \E_{O'_i}$ such that
\[ \big(\pi'(a) w_i - w_i \pi(f^*a)\big)\pi(h_i) \in \K_D(\E_W, \E_{W'}) \quad \text{for all } a \in C_0(W'). \]
To glue those isometries together and make it $\G_0$-equivariant, we shall use a standard Fell's absorption argument. Decompose $H_0 = \bigoplus_{i=1}^m H_{0,i}$ so that the ranges of $w_1, \dots, w_m$ lie in pairwise orthogonal summands, and identify $\E_{W'}$ with $\E_{W'} \ox_{C(B_0)} \ell^2(\G_0)$ via Fell's absorption theorem so that each fiber $(\E_{W'})_x$ acquires an extra tensor factor $\ell^2(r^{-1}(x))$. We define the operator $v_n : \E_W \to \E_{W'}$ fiberwise over $x \in B_0$ by
\[ (v_n \xi)_x := \sum_{g \in r^{-1}(x)} \sum_{i=1}^m \Big( V'_g \, w_i \, \pi(h_i) \, V_g^* \xi_x \Big) \ox \delta_g. \]
Since $\delta_g$ are pairwise orthogonal and $w_1, \dots, w_m$ have orthogonal ranges, we have
\[ \langle (v_n \xi)_x, (v_n \xi)_x \rangle = \sum_{g \in r^{-1}(x)} \sum_{i=1}^m \langle \xi_x, V_g \pi(h_i^2) V_g^* \xi_x \rangle = \Big\langle \xi_x, \pi\Big(\sum_{g, i} g \cdot h_i^2\Big)\xi_x \Big\rangle = \langle \xi_x, \xi_x \rangle, \]
which shows that $v_n$ is a $\G_0$-equivariant isometry. Moreover, since $\pi(h_i) V_g^* = V_g^* \pi(g \cdot h_i)$ and the range of $V'_g w_i$ is contained in $V'_g \E_{O'_i} = \overline{\pi'(C_0(g \cdot O'_i))\E_{W'}}$, each summand $V'_g w_i \pi(h_i) V_g^*$ is supported in $(g \cdot O'_i) \times_{B_0} (g \cdot O_i)$. Since the $\G_0$-action is fiberwise isometric, $\diam(g \cdot O'_i) = \diam(O'_i) < \delta_n$, and hence $d(y, f(w)) < \delta_n$ for any $(y, w) \in (g \cdot O'_i) \times_{B_0} (g \cdot O_i)$. Furthermore, for any $a \in C_c(W')$, properness of the $\G_0$-action implies that $(g \cdot O'_i) \cap \supp(a) \neq \emptyset$ for only a compact set of $g \in \G_0$. Thus, the sum for $\pi'(a)v_n - v_n\pi(f^*a)$ is locally finite and compact, and
\[
\supp(v_n) \subseteq \big\{ (y, x) \in W' \times_{B_0} W \;\big|\; d(y, f(x)) < \delta_n \big\}, \quad
\pi'(a) v_n - v_n \pi(f^*a) \in \K_D(\E_W, \E_{W'})
\]
for all $a \in C_0(W')$.

Using a standard trick as in \cite{Yu1997} by setting $V_f(n+s) = \cos(\pi s/2)v_n + \sin(\pi s/2)v_{n+1}$ for $s \in [0,1]$ with placing $v_n$ in orthogonal summands of $H_0$ defines a uniformly norm-continuous path of isometries $V_f(t) : \E_W \to \E_{W'}$ with covering scale $\delta(t) = \delta_{[ t ]} \to 0$. For any $T \in C^*_L(W, D)^{\G_0}$, (1) ensure that $V_f(t)T(t)V_f(t)^*$ is locally compact and satisfies
\[
\prop\big(V_f(t) T(t) V_f(t)^*\big) \le 2\delta(t) + \omega_f\big(\prop(T(t))\big) \longrightarrow 0 \quad \text{as } t \to \infty.
\]
Thus $\Ad_{V_f}$ defines a $*$-homomorphism from $C^*_L(W, D)^{\G_0}$ to $C^*_L(W', D)^{\G_0}$. Any two such paths $V_f$ and $\wt{V}_f$ are connected by the rotation homotopy $\cos(\pi s/2)V_f(t) \oplus \sin(\pi s/2)\wt{V}_f(t)$, showing that $f_* := (\Ad_{V_f})_*$ is independent of choices and functorial.
\end{proof}

We next prove the Mayer--Vietoris sequence for localization algebras. For any $\G_0$-invariant subcomplex $Z \subseteq W$ and $r > 0$, denote by $N_r(Z) := \{y \in W \mid d_{\rho(y)}(y, Z) < r\}$ its fiberwise $r$-neighborhood. Let $I_{Z,L} \triangleleft C^*_L(W, D)^{\G_0}$ be the closed two-sided ideal generated by all paths $f \in C^*_L(W, D)^{\G_0}$ such that for all sufficiently large $t$, $\supp(f(t)) \subseteq N_{\varepsilon(t)}(Z) \times_{B_0} N_{\varepsilon(t)}(Z)$ for some function $\varepsilon(t) \to 0$ as $t \to \infty$. Note that the annihilator ideal $C_0([0,\infty), C^*(W, D)^{\G_0})$ is contained in $I_{Z,L}$.

\begin{Pro}\label{pro:MV_comparison}
Let $Z \subseteq W$ be a closed $\G_0$-invariant subcomplex.
\begin{enumerate}
    \item[(1)] The inclusion $i : Z \to W$ induces a canonical excision isomorphism
    \[e_Z : K_*\big(C^*_L(Z, D)^{\G_0}\big) \to K_*(I_{Z,L}).\]
    \item[(2)] If $W = W_1 \cup W_2$ is a union of two closed $\G_0$-invariant subcomplexes with $W_{12} := W_1 \cap W_2$, then $I_{W_1, L} + I_{W_2, L} = C^*_L(W, D)^{\G_0}$ and $I_{W_1, L} \cap I_{W_2, L} = I_{W_{12}, L}$. Consequently, there is a canonical six-term Mayer--Vietoris exact sequence:
    \[\begin{tikzcd}[column sep=small] K_0(C^*_L(W_{12}, D)^{\G_0}) \arrow[r] & K_0(C^*_L(W_1, D)^{\G_0}) \oplus K_0(C^*_L(W_2, D)^{\G_0}) \arrow[r] & K_0(C^*_L(W, D)^{\G_0}) \arrow[d] \\ K_1(C^*_L(W, D)^{\G_0}) \arrow[u] & K_1(C^*_L(W_1, D)^{\G_0}) \oplus K_1(C^*_L(W_2, D)^{\G_0}) \arrow[l] & K_1(C^*_L(W_{12}, D)^{\G_0}) \arrow[l] \end{tikzcd}\]
\end{enumerate}
\end{Pro}

\begin{proof}
(1) Let $V_i(t): \E_Z \to \E_W$ be a covering isometry path for $i : Z \to W$ as in \Cref{pro:functoriality_loc} (2). Conjugation by $V_i(t)$ maps $C^*_L(Z, D)^{\G_0}$ into $I_{Z,L}$ and induces $e_Z$. Conversely, by \Cref{lem:polyhedral_rips_charts} (2), there is a $\G_0$-equivariant neighborhood retraction $p_Z: U_Z \to Z$ such that $p_Z \circ i = \mathrm{id}_Z$ and $\sup_{y \in N_r(Z)} d_{\rho(y)}(y, p_Z(y)) \to 0$ as $r \to 0$. For any controlled path $f \in I_{Z,L}$, $\supp(f(t)) \subseteq N_{\varepsilon(t)}(Z) \subseteq U_Z$ for $t$ sufficiently large. Conjugating the tail of $f$ by a covering isometry path $V_{p_Z}(t)$ for $p_Z$ defines a homomorphism $K_*(I_{Z,L}) \to K_*(C^*_L(Z, D)^{\G_0})$. Since $p_Z \circ i = \mathrm{id}_Z$ on $Z$ and $\sup_{y \in N_{\varepsilon(t)}(Z)} d_{\rho(y)}(y, i(p_Z(y))) \to 0$ as $t \to \infty$, both compositions $V_{p_Z}(t)V_i(t)$ and $V_i(t)V_{p_Z}(t)$ cover the identity map with error tending to zero as $t \to \infty$, and hence induce the identity on $K$-theory.

(2) Since $W_1$ and $W_2$ are subcomplexes, there exists a function $a(r) \to 0$ as $r \to 0^+$ such that $N_r(W_1) \cap N_r(W_2) \subseteq N_{a(r)}(W_{12})$. Multiplying any path in $C^*_L(W, D)^{\G_0}$ by a $t$-continuous family of $\G_0$-invariant partition of unity subordinate to $N_{\varepsilon(t)}(W_1)$ and $N_{\varepsilon(t)}(W_2)$ shows that $I_{W_1, L} + I_{W_2, L} = C^*_L(W, D)^{\G_0}$. Moreover, $N_r(W_1) \cap N_r(W_2) \subseteq N_{a(r)}(W_{12})$ implies $I_{W_1, L} \cap I_{W_2, L} = \overline{I_{W_1, L} I_{W_2, L}} = I_{W_{12}, L}$. Combining the standard Mayer--Vietoris sequence for the ideals $I_{W_1, L}, I_{W_2, L}$ with the isomorphism in (1) yields the desired six-term exact sequence.
\end{proof}

\begin{Pro}\label{pro:homotopy_invariance_loc}
Let $W$ and $W'$ be proper, cocompact metric $\G_0$-spaces. If $f_0, f_1 : W \to W'$ are joined by a continuous, $\G_0$-equivariant homotopy, then $(f_0)_* = (f_1)_*$ on $K_*(C^*_L(W, D)^{\G_0})$.
\end{Pro}

\begin{proof}
Consider the $W \times [0,1]$ equipped with the trivial $\G_0$-action on $[0,1]$, so that $\E_{W \times [0,1]} \cong \E_W \ox L^2([0,1])$. Since $\G_0$ acts trivially on the $[0,1]$-coordinate, the standard Eilenberg swindle and Mayer--Vietoris argument along the interval factor (see \cite[Section 6.4]{HIT2020}) take place entirely in $\L(L^2([0,1]) \ox H_0)$ and hence automatically preserve $\G_0$-equivariance and $D$-linearity. It follows that the inclusion $i_t : W \to W \times [0,1]$, $w \mapsto (w, t)$, induces the same isomorphism $(i_t)_* : K_*(C^*_L(W, D)^{\G_0}) \to K_*(C^*_L(W \times [0,1], D)^{\G_0})$ for all $t \in [0,1]$. Applying \Cref{pro:functoriality_loc} (2) to the homotopy $H : W \times [0,1] \to W'$ then yields $(f_0)_* = H_*\circ (i_0)_* = H_*\circ(i_1)_* = (f_1)_*$.
\end{proof}

One can check that the functoriality (\Cref{pro:functoriality_loc}), the Mayer--Vietoris exact sequence (\Cref{pro:MV_comparison}), and homotopy invariance (\Cref{pro:homotopy_invariance_loc}) are naturally compatible with the comparison map $\eta$. We are now ready to prove the main theorem of this subsection.

\begin{Thm}\label{thm:loc_to_KK}
For every $d \ge 0$, the map
\[\eta_{X_d} : K_*\big(C^*_L(X_d, \ell^\infty(X, \K)^\Ga)^{G(X,\Ga)}\big) \to KK_*^{G(X,\Ga)}\big(C_0(X_d), \ell^\infty(X, \K)^\Ga\big)\]
is an isomorphism.
\end{Thm}

\begin{proof}
By \Cref{lem:countable_data_reduction}, it suffices to prove that $\eta_Y : K_*(C^*_L(Y, D)^{\G_0}) \to KK^{\G_0}_*(C_0(Y), D)$ is an isomorphism for the countable model $Y = Y_d$ over $\G_0$ and any separable $\G_0$-invariant $C(B_0)$-subalgebra $D \subseteq A$. We divide the proof into three steps and the idea follows a standard argument of \cite[Theorem 3.2]{Yu2000} in the groupoid equivariant setting.

\noindent\emph{Step 1. Reduction to zero-dimensional subcomplexes.} 
Following Yu's cutting-and-pasting argument on simplicial complexes, we prove by induction on $k \in \{0, 1, \dots, N\}$, where $N = \dim Y$.

Suppose $k \ge 1$ and the claim holds for all closed $\G_0$-invariant subcomplexes of dimension at most $k-1$. We denote $W^{(k)} \subseteq W$ to be the $k$-skeleton of $W$ for each $k$. In each $k$-simplex $\sigma = [v_0, \dots, v_k]\subseteq W^{(k)}$ with barycenter $b_\sigma$, any point $y = \sum_{j=0}^k a_j v_j \in \sigma$ can be written in radial coordinates from $b_\sigma$ to the boundary $\partial \sigma \subseteq W^{(k-1)}$ as
\[ y = \big(1 - \tau(y)\big) b_\sigma + \tau(y) z. \]
Note that $\tau$ is a continuous $\G_0$-invariant function on $W^{(k)}$. Passing to a suitable equivariant polyhedral subdivision, we decompose $W^{(k)}$ into two closed $\mathcal{G}_0$-invariant subcomplexes
\[ W^{(k)}_1 = \big\{ y \in W^{(k)} \setminus W^{(k-1)} \;\big|\; \tau(y) \le 2/3 \big\}, \quad W^{(k)}_2 = W^{(k-1)} \cup \big\{ y \in W^{(k)} \setminus W^{(k-1)} \;\big|\; \tau(y) \ge 1/3 \big\}. \]
The linear homotopy retracts $W^{(k)}_1$ to the closed, $\G_0$-invariant, zero-dimensional subcomplex $T_k = \{b_\sigma \mid \dim \sigma = k\}$ of barycenters of $k$-simplices, retracts $W_2^{(k)}$ to the $(k-1)$-skeleton $W^{(k-1)}$. The intersection $W^{(k)}_{12}$ retracts again to a closed $\G_0$-invariant complex of dimension $k-1$. By the homotopy invariance of both functors, the induction hypothesis implies that $\eta_{W_1^{(k)}}$, $\eta_{W_2^{(k)}}$ and $\eta_{W^{(k)}_{12}}$ are isomorphisms. Applying the Five Lemma and the Mayer--Vietoris exact sequences of both functors shows that $\eta_W$ is an isomorphism. This reduces the theorem to the $0$-dimensional case.

\noindent\emph{Step 2. Reduction to slices.} 
Let $T $ be a $0$-dimensional complex. By the slice theorem for proper \'{e}tale groupoids \cite[Proposition 5.30]{MM2003}, each point $t \in T$ has a finite stabilizer group $F$ and a clopen neighborhood $U$ invariant under $F$-action such that the orbit of $U$ is homeomorphic to $\G_0 \times_{F \ltimes U} U$. Since $T/\G_0$ is compact, there exists a finite cover of $T$ by such slices
\begin{equation}\label{eq:zero_dim_slice_decomp}
T = \bigsqcup_{j=1}^m T_j, \qquad T_j \cong \G_0 \times_{\H_j} U_j, \qquad \H_j = F_j \ltimes U_j,
\end{equation}
where each $U_j \subseteq B_0$ is clopen and $F_j$ is a finite group. 

For each piece $T_j = \G_0 \times_{\H_j} U_j$, set $D_{U_j} = C(U_j)D$ and let $P = \pi(\chi_{U_j})$ be the projection of the clopen subset $U_j \subseteq T_j$ on $\E_{T_j}$. By Kasparov's equivariant stabilization theorem, we have an $\H_j$-equivariant unitary $(\ell^2(\G_0)|_{U_j}) \ox H_0 \cong \ell^2(\H_j) \ox H_0$ since $F_j$ is finite, which identifies $P\E_{T_j}$ with $\E_{U_j}^{\H_j}$. At the operator level, we define the spatial restriction map and the block induction map
\begin{align*}
& R : C^*(T_j, D)^{\G_0} \to C^*(U_j, D_{U_j})^{\H_j}, && R(T) = P T P, \\
& I : C^*(U_j, D_{U_j})^{\H_j} \to C^*(T_j, D)^{\G_0}, && I(S) = \bigoplus_{[g] \in (\G_0)_{U_j} / \H_j} V_g S_{s(g)} V_g^*,
\end{align*}
Because $U_j$ is compact and the metric is proper, distinct points in each fiber of $T_j$ are separated by some $b > 0$, any path $h \in C^*_L(T_j, D)^{\G_0}$ becomes block-diagonal once $\prop(h(t)) < b$, so that $\lim_{t\to\infty}\|h(t) - I(R(h))(t)\| = 0$. Thus $R$ and $I$ are mutually inverse modulo the annihilator ideal $C_0([0,\infty), C^*(\cdot))$, thus induce a canonical isomorphism $Res_L : K_*(C^*_L(T_j, D)^{\G_0}) \to K_*(C^*_L(U_j, D_{U_j})^{\H_j})$. Restriction also acts as isomorphism on $KK^{\G_0}$-theory,  by Shapiro's lemma, we have the following diagram
\[ \begin{tikzcd}[column sep=large] K_*\big(C^*_L(T_j, D)^{\G_0}\big) \arrow[r, "\eta_{T_j}"] \arrow[d, "Res_L"', "\cong"] & KK^{\G_0}_*\big(C_0(T_j), D\big) \arrow[d, "Res_{KK}", "\cong"'] \\ K_*\big(C^*_L(U_j, D_{U_j})^{\H_j}\big) \arrow[r, "\eta_{U_j}^{\H_j}"] & KK^{\H_j}_*\big(C(U_j), D_{U_j}\big) \end{tikzcd} \]

\noindent\emph{Step 3. Local isomorphism on slices.} 
When the coefficient algebra is trivial, the isomorphism $\eta_{U_j}^{\H_j}$ was proved in \cite{BenameurRoy2021}. For a general separable coefficient algebra $D_{U_j}$, notice that over the space $U_j$ and groupoid $\H_j = F_j \ltimes U_j$, every fiber is a single point. Thus, the propagations are always zero and $C^*_L(U_j, D_{U_j})^{\H_j} = \mathfrak{T}C^*(U_j, D_{U_j})^{\H_j}$, which implies that the evaluation map $ev_* : K_*(C^*_L(U_j, D_{U_j})^{\H_j}) \to K_*(C^*(U_j, D_{U_j})^{\H_j})$ is an isomorphism by the same backward-shift Eilenberg swindle as in \Cref{lem:evaluation_iso}. On the other hand, since $F_j$ is a finite group, $\H_j = F_j \ltimes U_j$ is an amenable \'{e}tale groupoid, and $U_j$ serves as its classifying space $\underline{E}\H_j$. By Tu's theorem \cite{Tu1999}, the groupoid Baum--Connes assembly map with coefficients in $D_{U_j}$,
\[\mu_{D_{U_j}}^{\H_j} : KK^{\H_j}_*\big(C(U_j), D_{U_j}\big) \to K_*\big(D_{U_j} \rtimes_r \H_j\big),\]
is an isomorphism. Combining \Cref{thm:groupoid Roe algebra} (2) and \Cref{thm:localization algebra}, we have $\M_* \circ ev_* = \mu_{D_{U_j}}^{\H_j} \circ \eta_{U_j}^{\H_j}$. Since $\M_*$, $ev_*$, and $\mu_{D_{U_j}}^{\H_j}$ are all isomorphisms, $\eta_{U_j}^{\H_j}$ is an isomorphism. Taking the direct sum over $1 \le j \le m$ proves that $\eta_T$ (and hence $\eta_Y$) is an isomorphism.
\end{proof}

\subsection{Proof of the main theorem}
We are now ready to complete the proof of the main equivalence theorem of this paper.

\begin{Thm}\label{thm:final_equivalence}
Let $X$ be a metric space with bounded geometry, and $\Gamma$ a countable discrete group acting properly, freely, and isometrically on $X$. The equivariant coarse Baum--Connes conjecture holds for $(X,\Gamma)$ if and only if the groupoid Baum--Connes conjecture holds for $G(X,\Gamma)$ with coefficients in $\ell^\infty(X,\K)^\Gamma$.
\end{Thm}

\begin{proof}
Let $Z$ be a proper, cocompact topological $G(X,\Gamma)$-space with anchor map $\rho_Z: Z \to \beta(X/\Gamma)$. We shall construct a continuous, $G(X,\Gamma)$-equivariant map $F: Z \to X_d$ for sufficiently large $d$.

Since the continuous action of $G(X,\Gamma)$ on $Z$ is proper and cocompact, there exists a relatively compact open subset $U \subseteq Z$ such that $Z = G(X,\Gamma) \cdot U$. By the standard theory of proper groupoid actions, there exists a continuous cut-off function $c: Z \to [0, \infty)$ supported in $U$ such that $\sum_{g\in r^{-1}(\rho_Z(z))}c(g^{-1}\cdot z)=1$.

We define the map $F: Z \to X_d$ by sending $z \in Z$ to the formal convex combination
\[ F(z) := \sum_{g \in r^{-1}(\rho_Z(z))} c(g^{-1} \cdot z) g.\]
We first verify that the image of $F$ is indeed contained in $X_d$ for some $d \ge 0$. Set $K=\supp(c)$. If $g_1, g_2 \in r^{-1}(\rho_Z(z))$ yield non-zero coefficients in $F(z)$, then $g_1^{-1} \cdot z \in K$ and $g_2^{-1} \cdot z \in K$. This implies that the groupoid element $g_1^{-1}g_2$ transports the point $g_2^{-1}\cdot z \in K$ to $g_1^{-1}\cdot z \in K$. Because the action is proper, the subset 
\[ \S_K = \{ \gamma \in G(X,\Gamma) \mid \gamma \cdot K \cap K \neq \emptyset \} \]
is compact in $G(X,\Gamma)$. By the definition of the coarse groupoid, $\S_K$ is contained in $\overline{\Delta_d/\Gamma}$ for some $d \ge 0$. This ensures that the distance between any two vertices $g_1, g_2$ with non-zero coefficients is uniformly bounded by $d$. Hence, the non-zero terms in $F(z)$ precisely span a simplex in the Rips complex structure of $X_d$. 

The $G(X,\Gamma)$-equivariance of $F$ follows by a direct computation. Let $\gamma \in G(X,\Gamma)$ act on $z \in Z$, where $s(\gamma) = \rho_Z(z)$. Using the bijection $ g \mapsto h=\gamma g$ from $r^{-1}(\rho_Z(z))$ to $r^{-1}(r(\gamma))$, we have:
\[ F(\gamma \cdot z) = \sum_{h \in r^{-1}(r(\gamma))} c(h^{-1} \cdot (\gamma \cdot z)) h = \sum_{g \in r^{-1}(\rho_Z(z))} c(g^{-1} \cdot z) \gamma g = \gamma \cdot \Big( \sum_{g} c(g^{-1} \cdot z) g \Big). \]
This shows that $F$ is strictly equivariant. Since the cutoff function $c(z)$ is globally continuous on $Z$, it is straightforward to check that $F$ is well-defined and continuous.

This construction shows that any proper cocompact $G(X,\Gamma)$-space admits a continuous equivariant map into some $X_d$. Moreover, for any two continuous $G(X,\Ga)$-equivariant maps $F_0, F_1: Z \to X_d$, their vertex supports must lie in $\overline{\Delta_R/\Gamma}$ for some $R \ge 0$. Since $\overline{\Delta_R/\Gamma}^{-1}\overline{\Delta_R/\Gamma} \subseteq \overline{\Delta_{2R}/\Gamma}$, taking $d' \ge \max\{d, 2R\}$ ensures that the linear interpolation $(1-t)F_0 + tF_1$ defines a $G(X,\Gamma)$-equivariant homotopy in $X_{d'}$.

Therefore, the direct limit over $\{X_d\}_{d \ge 0}$ computes the $K$-homology of the classifying space of $G(X,\Ga)$. Combining this with the canonical isomorphisms and commutative diagrams in \Cref{lem:K_homology_iso}, \Cref{thm:loc_to_KK}, and \Cref{cor: equivariant Roe and groupoid}, the assembly maps for the equivariant coarse Baum--Connes conjecture and the groupoid Baum--Connes conjecture strictly commute. This completes the proof.
\end{proof}

\section{Applications to higher index theory}

In this section, we shall use our groupoid model to study higher index problems.

\subsection{A-T-menability of the equivariant coarse groupoid}\label{sec: a-T-manable}

In this subsection, we study the analytic properties of the equivariant coarse groupoid. In particular, we establish an equivalence between equivariant coarse embeddability into Hilbert space and the a-T-menability of the associated equivariant coarse groupoid.

We first recall the definition of a-T-menability for \'etale groupoids $\G$. A continuous function $\psi: \G \to \IR$ on a groupoid $\G$ is called \emph{conditionally negative definite} if $\psi(g) = \psi(g^{-1})$, $\psi(u) = 0$ for all $u \in \G^{(0)}$, and for any $u \in \G^{(0)}$, any finite set of elements $g_1, \dots, g_n \in s^{-1}(u)$, and any real numbers $c_1, \dots, c_n$ such that $\sum_{i=1}^n c_i = 0$, we have
$$ \sum_{i,j=1}^n c_i c_j \psi(g_i g_j^{-1}) \le 0. $$
A groupoid admitting a proper conditionally negative definite function is said to be \emph{a-T-menable}.

Now we are ready to prove the main theorem of this section.

\begin{Thm}\label{thm: a-T-menable}
Let $X$ be a bounded geometry metric space equipped with a free, proper, and isometric right $\Ga$-action. The space $X$ admits an equivariant coarse embedding into Hilbert space if and only if its equivariant coarse groupoid $G(X, \Ga)$ is a-T-menable.
\end{Thm}

\begin{proof}
$(\Longrightarrow)$ Suppose $f: X \to \mathcal{H}$ is an equivariant coarse embedding into Hilbert space equipped with a right affine isometric action of $\Ga$. Let $U_\gamma \in \mathcal{O}(\mathcal{H})$ and $v_\gamma \in \mathcal{H}$ denote the linear and translation parts of the affine action, respectively.

We define a function on the discrete subset $(X \times X)/\Ga$ by $\psi_0([x, y]) = \|f(x) - f(y)\|^2$. The $\Ga$-invariance guarantees that $\psi_0$ is well-defined. Since $f$ is a coarse embedding, $\psi_0$ is bounded on each $\Delta_R/\Ga$, and extends uniquely on each compact open set $\overline{\Delta_R/\Ga}$ to a proper continuous function $\psi: G(X,\Ga) \to [0,\infty)$.

For the conditional negative definiteness, consider a unit $[x]$ and elements $g_i = [y_i, x]$. For any $c_1, \dots, c_n \in \IR$ with $\sum c_i = 0$, we have
\begin{align*}
    \sum_{i,j=1}^n c_i c_j \psi(g_i g_j^{-1}) &= \sum_{i,j=1}^n c_i c_j \|f(y_i) - f(y_j)\|^2\le 0.
\end{align*}
The inequality directly follows from Schoenberg's theorem, see, for instance, \cite[Theorem C.2.3]{KazhdanT}. The inequality extends by continuity from the dense discrete subset $(X \times X)/\Ga$ to the entire compactification $G(X, \Ga)$. 

Therefore, $\psi$ is a proper, continuous, conditionally negative definite function on $G(X, \Ga)$, proving that the groupoid is a-T-menable.

$(\Longleftarrow)$ Conversely, assume $G(X, \Ga)$ is a-T-menable, so there exists a proper, continuous, conditionally negative definite function $\psi: G(X, \Ga) \to [0, \infty)$. We restrict $\psi$ to the dense discrete subgroupoid $(X \times X)/\Ga$ and define a symmetric kernel $k: X \times X \to [0, \infty)$ by $k(x, y) = \psi([x, y])$.

By definition, $k(x\gamma, y\gamma) = k(x, y)$ for all $\gamma \in \Ga$, and $k$ is a conditionally negative definite kernel on $X$. By the standard GNS construction (or Schoenberg's theorem, see \cite[Theorem C.2.3]{KazhdanT}), there exists a real Hilbert space $\mathcal{H}$ and a map $f: X \to \mathcal{H}$ such that $\|f(x) - f(y)\|^2 = k(x, y)$. The $\Ga$-invariance of the kernel implies that the map $f(x) \cdot \gamma := f(x\gamma)$ preserves distances in $f(X)$. By construction, $f$ is strictly $\Ga$-equivariant.

Finally, since $\psi$ is a continuous and proper function on $G(X, \Ga)$, it is straightforward to verify that $f$ is a coarse embedding. This completes the proof.
\end{proof}

As a direct consequence of \Cref{thm: a-T-menable} and the profound work of J.~L.~Tu \cite{Tu1999}, we obtain the following geometric application, which recovers a main result of B.~Fu and X.~Wang \cite{FW2016} from a purely groupoid perspective.

\begin{Cor}[\cite{FW2016}]\label{cor:BC_conjecture}
Let $X$ be a bounded geometry metric space equipped with a free, proper, and isometric right $\Ga$-action. If $X$ admits an equivariant coarse embedding into Hilbert space, then its equivariant coarse groupoid $G(X, \Ga)$ satisfies the Baum--Connes conjecture with coefficients. Consequently, the equivariant coarse Baum--Connes conjecture holds for $X$.\qed
\end{Cor}

The fact that equivariant coarse embeddability into Hilbert space implies the equivariant coarse Baum--Connes conjecture was originally proved by B.~Fu and X.~Wang \cite{FW2016} using the machinery of equivariant localization algebras and Dirac-dual-Dirac methods.  As clarified in \cite{GWWY2024}, their result is naturally formulated under the assumption of equivariant bounded geometry, which is equivalent to the assumption that the action is equivariantly coarsely equivalent to a free action. Thus, the free action setting considered in the present paper is precisely the natural framework of the previous result, and our proof provides a new groupoid-equivariant approach to this case.

\subsection{Equivariant coarse Novikov conjecture}

In this section, we apply the groupoid framework to provide a geometric proof of the equivariant coarse Novikov conjecture for a proper metric space $X$ equipped with a proper, free, and isometric right $\Gamma$-action. Crucially, we only assume that $X$ admits a coarse embedding into Hilbert space (which need not be $\Gamma$-equivariant). The main theorem of this section is stated as follows.

\begin{Thm}\label{thm: ECNC for CE spaces}
Let $X$ be a metric space with bounded geometry, $\Ga$ a countable discrete group acting properly and isometrically on $X$. If $X$ admits a coarse embedding into Hilbert space and the action $X\cal\Ga$ has equivariant bounded geometry, then the Baum--Connes assembly map for $G(X,\Ga)$ with any coefficient $A$ is an injection.

As a result, the equivariant coarse Novikov conjecture holds for $X\cal\Ga$.
\end{Thm}

Without loss of generality, we shall assume the action $X\cal\Ga$ is free.

\begin{Lem}\label{lem:groupoid_isomorphism}
Let $p: \beta X \to \beta(X/\Gamma)$ be the continuous extension of the quotient map $p: X\to X/\Ga$. The space $\beta X$ admits a continuous left action by the equivariant coarse groupoid $G(X, \Gamma)$ with proper anchor map $p$. Furthermore, we have a topological groupoid isomorphism:
\[ G(X) \cong G(X, \Gamma) \ltimes \beta X. \]
\end{Lem}

\begin{proof}
We first define a left action of $G(X, \Gamma)$ on $\beta X$. On the dense subset $X$, for any $g = [y, x] \in \Delta_R / \Gamma$ and $z \in X$ such that $p(z) = s(g)$, we have $[z] = [x]$. Since $\Gamma$ acts freely, there is a unique $\gamma \in \Gamma$ such that $z = x \gamma$. We define the left action as $g \cdot z = y \gamma$. 

Choose another representative $(y \delta, x \delta)$ for $g$. Then we can write $z = (x \delta)(\delta^{-1} \gamma)$, which yields the action $(y \delta)(\delta^{-1} \gamma) = y \gamma$. This shows that this action is well-defined. Since the right $\Gamma$-action on $X$ is isometric, $d_X(g \cdot z, z) = d_X(y \gamma, x \gamma) = d_X(y, x) \le R$. Thus, the action transfers points with a uniformly bounded amount. By the universal property of the Stone-\v{C}ech compactification, this action extends to a continuous groupoid action of $G(X, \Gamma)$ on $\beta X$ by \Cref{lem:key lemma}.

We form the transformation groupoid $G(X, \Gamma) \ltimes \beta X$. On the discrete level, the map $\Phi([y, x], x \gamma) = (y \gamma, x \gamma)$ is a well-defined bijection from $\Delta_R/\Gamma \times_{X/\Gamma} X$ to $\Delta_R$. Extending this to the compactifications yields the desired isomorphism
$$\Phi: G(X, \Gamma) \ltimes \beta X \to G(X),$$
as desired.
\end{proof}

For any $G(X,\Ga)$-$C^*$-algebra $A$, we shall denote $A_{\beta X}=C(\beta X)\ox_{C(\beta(X/\Gamma))}A$ for simplicity.

\begin{Cor}\label{lem:descent_bc}
Assume that $X$ admits a coarse embedding into a Hilbert space. Then the equivariant coarse groupoid $G(X, \Gamma)$ satisfies the groupoid Baum--Connes conjecture with coefficients in $A_{\beta X}$, i.e., the assembly map
\[ \mu_{A_{\beta X}}: K_*^{top}(G(X, \Gamma), A_{\beta X}) \to K_*(A_{\beta X} \rtimes_r G(X, \Gamma)) \]
is an isomorphism.
\end{Cor}

\begin{proof}
Because $X$ admits a coarse embedding into a Hilbert space, the coarse groupoid $G(X)$ is a-T-menable by the foundational results of Skandalis, Tu, and Yu \cite{STY2002}. Consequently, by Tu's theorem \cite{Tu1999}, the groupoid $G(X)$ satisfies the groupoid Baum--Connes conjecture with any coefficients.

Under the canonical isomorphism $G(X) \cong G(X, \Gamma) \ltimes \beta X$ established in \Cref{lem:groupoid_isomorphism}, we can apply the Shapiro Lemma for groupoids (see, e.g., \cite[Lemma 4.1]{STY2002}). It implies that the Baum--Connes assembly map for $G(X, \Gamma)$ with coefficients in $A_{\beta X}$ is canonically identified with the assembly map for the transformation groupoid $G(X, \Gamma) \ltimes \beta X \cong G(X)$ with coefficients in $A_{\beta X}$. Since the latter is an isomorphism, the Baum--Connes assembly map for $G(X, \Gamma)$ with coefficients in $A_{\beta X}$ is also an isomorphism.
\end{proof}

For any locally compact, Hausdorff space $X$, we shall denote $M_1(X)$ the space of all regular Borel probability measures on $X$ equipped with the weak-$*$ topology. For any $x\in X$, we denote $\delta_x$ the Dirac measure on $x$.

Define $\mathcal{M}$ to be the space of regular Borel probability measures on $\beta X$ that are supported entirely on a single fiber of $p: \beta X\to\beta(X/\Ga)$. Explicitly, using the continuous push-forward map $p_*: M_1(\beta X) \to M_1(\beta(X/\Gamma))$, we define
\[\M = \{ \mu \in M_1(\beta X) \mid p_*(\mu) = \delta_\omega \text{ for some } \omega \in \beta(X/\Gamma) \}. \]
Since the set of Dirac measures $\{ \delta_\omega \mid \omega \in \beta(X/\Gamma) \}$ is closed in $M_1(\beta(X/\Gamma))$, which can be identified with $\beta (X/\Ga)$, $\M$ is actually identified with $p_*^{-1}(\beta(X/\Ga))\subseteq M_1(\beta X)$. Therefore, it is also a compact Hausdorff space. We have a canonical continuous projection $q: \mathcal{M} \to \beta(X/\Gamma)$ given by $q(\mu) = \omega$, where $p_*(\mu) = \delta_\omega$. For each $\omega \in \beta(X/\Gamma)$, the fiber $\mathcal{M}_\omega = q^{-1}(\omega)$ is exactly $M_1(p^{-1}(\omega))$, the space of probability measures on the compact fiber $p^{-1}(\omega)$. 

The anchor map for the left action of $G(X, \Gamma)$ on $\beta X$ is exactly the continuous projection $p: \beta X \to \beta(X/\Gamma)$. Therefore, any groupoid element $g \in G(X, \Gamma)$ with source $s(g) = \omega$ and range $r(g) = \eta$ restricts to a homeomorphism between the corresponding compact fibers
\[ \alpha_g: p^{-1}(\omega) \to p^{-1}(\eta), \quad x \mapsto g \cdot x. \]
This geometric homeomorphism canonically induces a left action of $G(X, \Gamma)$ on the fiberwise measure space $\mathcal{M}$ via the push-forward of measures. For a probability measure $\mu \in \mathcal{M}_\omega = M_1(p^{-1}(\omega))$, we define $g \cdot \mu := g_* \mu \in \mathcal{M}_\eta$.

To see that this induced groupoid action is continuous, recall that $\mathcal{M}$ inherits the weak-$*$ topology as a closed subspace of $M_1(\beta X)$. Since the original action $(g, x) \mapsto g \cdot x$ is continuous on the transformation groupoid $G(X, \Gamma) \ltimes \beta X$, the composed function $(g, x) \mapsto f(g \cdot x)$ is continuous on the fibered product $G(X,\Ga)\times_{\beta(X/\Ga)}\beta X=\{(g,x)\mid s(g)=p(x)\}$ for any $f \in C(\beta X)$. Consequently, the integral mapping
\[ (g, \mu) \mapsto \int_{\beta X} f(g \cdot x) \, d\mu(x) \]
is continuous on the fibered product $G(X,\Ga)\times\M$ with respect to the groupoid topology on $G(X, \Gamma)$ and the weak-$*$ topology on $\mathcal{M}$ for any $f\in C(\beta X)$. This shows that the push-forward map $(g, \mu) \mapsto g_* \mu$ is a continuous groupoid action. Thus, we have the topological transformation groupoid $G(X, \Gamma) \ltimes \mathcal{M}$.

Because $X$ admits a coarse embedding into Hilbert space, the coarse groupoid $G(X) \cong G(X, \Gamma) \ltimes \beta X$ is a-T-menable \cite{STY2002}. Let $\psi: G(X, \Gamma) \ltimes \beta X \to [0, \infty)$ be a proper, continuous, conditionally negative definite function. We define a function on $G(X, \Gamma) \ltimes \mathcal{M}$ by:
\[ \widetilde{\psi}(g, \mu) = \int_{\beta X} \psi(g, x) d\mu(x). \]
To verify that $\widetilde{\psi}$ is conditionally negative definite, we take $\omega \in \beta(X/\Gamma)$, $g_1, \dots, g_n \in G(X, \Gamma)$ share the same source $\omega$, $\mu \in \mathcal{M}_\omega$, and $c_1, \dots, c_n \in \mathbb{R}$ with $\sum_{i=1}^n c_i = 0$. Since the inverse of $(g_j, \mu)$ is $(g_j^{-1}, g_{j*} \mu)$, we compute:
\[ \sum_{i,j=1}^n c_i c_j \widetilde{\psi}(g_i g_j^{-1}, g_{j*} \mu) = \int_{p^{-1}(\omega)} \left( \sum_{i,j=1}^n c_i c_j \psi(g_i g_j^{-1}, g_j x) \right) d\mu(x). \]
Because $\psi$ is conditionally negative definite on $G(X, \Gamma) \ltimes \beta X$, the integrand is non-positive for every $x$. Thus, the integral is $\le 0$.

Furthermore, the function $\widetilde{\psi}$ is also proper. By definition, we need to show that for any $C > 0$, the sublevel set
$$ \widetilde{K}_C = \{ (g, \mu) \in G(X, \Gamma) \ltimes \mathcal{M} \mid \widetilde{\psi}(g, \mu) \le C \} $$
is compact. Since the original function $\psi$ is proper on $G(X, \Gamma) \ltimes \beta X$, its corresponding sublevel set $K_C = \{ (g, x) \mid \psi(g, x) \le C \}$ is a compact subspace. Consider the canonical projection map $\pi_1: G(X, \Gamma) \ltimes \beta X \to G(X, \Gamma)$ given by $(g, x) \mapsto g$, which is a continuous map. Thus, the set $\pi_1(K_C)$ is a compact subset of $G(X, \Gamma)$.

Now, suppose $(g, \mu) \in \widetilde{K}_C$. By definition, the integral yields $\int_{p^{-1}(s(g))} \psi(g, x) \, d\mu(x) \le C$. Recall that $\mu$ is a probability measure. Thus, there exists at least one point $x_0 \in \mathrm{supp}(\mu) \subseteq p^{-1}(s(g))$ such that $\psi(g, x_0) \le C$. This implies that the pair $(g, x_0)$ belongs to the compact set $K_C$. Therefore, if $(g, \mu) \in \widetilde{K}_C$, the groupoid element $g$ must belong to the compact set $\pi_1(K_C)$. This means the sublevel set $\widetilde{K}_C$ is contained in the restricted fibered product $\pi_1(K_C)\ltimes \mathcal{M}$. Therefore, $\wt{K_C}$ is compact. This concludes the proof that $\widetilde{\psi}$ is proper.

Thus, $G(X, \Gamma) \ltimes \mathcal{M}$ is a-T-menable. For any $G(X,\Ga)$-$C^*$-algebra $A$, we shall denote $A_{\M}=C(\M)\ox_{C(\beta(X/\Gamma))}A$ for simplicity.
By Shapiro's Lemma again, we further have the following corollary.

\begin{Cor}\label{cor:descent_bc_M}
For any $G(X,\Ga)$-$C^*$-algebra $A$, the groupoid Baum--Connes assembly map for $G(X, \Gamma)$ with coefficients in $A_{\M}$ is an isomorphism:
\[ \mu_{\M}: K_*^{top}(G(X, \Gamma), A_{\M}) \xrightarrow{\cong} K_*(A_{\M} \rtimes_r G(X, \Gamma)). \]
\end{Cor}

Similarly, the projection $p_*: \M \to \beta(X/\Gamma)$ induces a $G(X, \Gamma)$-equivariant injective $\iota: C(\beta(X/\Ga)) \to C(\M)$. Then for any $G(X,\Gamma)$-$C^*$-algebra $A$, this induces a commutative diagram:
\begin{equation}\label{eq:iota}
\begin{tikzcd}
K_*^{\mathrm{top}}(G(X, \Gamma), A) \arrow[r, "\mu"] \arrow[d, "\iota_{*}^{\mathrm{top}}"'] & K_*(A\rtimes_r G(X, \Gamma)) \arrow[d, "\iota_{*}^{\mathrm{an}}"] \\
K_*^{\mathrm{top}}(G(X, \Gamma), A_{\M}) \arrow[r, "\mu_{\M}"] & K_*(A_{\M} \rtimes_r G(X, \Gamma))
\end{tikzcd}
\end{equation}
To prove that $\mu$ is injective, it suffices to prove that the left vertical map $\iota_{*}^{\mathrm{top}}$ is an isomorphism since $\mu_{\M}$ is an isomorphism by \Cref{cor:descent_bc_M}. 

\begin{Lem}\label{lem:iota}
With notations as above, the map $\iota_{*}^{\mathrm{top}}$ in \eqref{eq:iota} is an isomorphism.
\end{Lem}

\begin{proof}
\emph{Step 1. Reduction to zero-dimensional slices.} 
By \Cref{thm:final_equivalence}, it suffices to show that for each $d \ge 0$, the induced map
\[ \iota_*: KK^{G(X, \Gamma)}_*\big(C_0(X_d), A\big) \longrightarrow KK_*^{G(X, \Gamma)}\big(C_0(X_d), A_{\M}\big) \]
is an isomorphism. By the Mayer--Vietoris induction and zero-dimensional slice decomposition established in the proof of \Cref{thm:loc_to_KK} (Steps 1 and 2, pulling back the clopen slices from $B_0$ to $\beta(X/\Gamma)$), the problem reduces to zero-dimensional $G(X, \Gamma)$-invariant slices of the form
\[ V_j \cong G(X, \Gamma) \times_{\H_j} U_j, \qquad \H_j = F_j \ltimes U_j, \]
where $U_j \subseteq \beta(X/\Gamma)$ is a clopen subset and $F_j$ is a finite group of clopen bisections of $G(X, \Gamma)$ acting continuously on $U_j$.

By the same Shapiro lemma as in Step 2 of the proof of \Cref{thm:loc_to_KK}, evaluating $G(X, \Gamma)$-equivariant $KK$-theory over the induced space $V_j = G(X, \Gamma) \times_{\H_j} U_j$ reduces canonically to evaluating $\H_j$-equivariant $KK$-theory over $U_j$
\begin{align*}
    KK_*^{G(X, \Gamma)}\big(C_0(V_j), A\big) &\cong KK_*^{\H_j}\big(C(U_j), A_{U_j}\big), \\
    KK_*^{G(X, \Gamma)}\big(C_0(V_j), A_\M\big) &\cong KK_*^{\H_j}\big(C(U_j), A_{p_*^{-1}(U_j)}\big),
\end{align*}
where $p_*: \M \to \beta(X/\Ga)$ is the canonical projection, $A_{U_j} = C(U_j) \otimes_{C(\beta(X/\Ga))} A$, and $A_{p_*^{-1}(U_j)} = C(p_*^{-1}(U_j)) \otimes_{C(U_j)} A_{U_j}$. Thus, it suffices to show the following inclusion induces an isomorphism
\[ \iota^{U_j}_*: KK_*^{\H_j}\big(C(U_j), A_{U_j}\big) \longrightarrow KK_*^{\H_j}\big(C(U_j), A_{p_*^{-1}(U_j)}\big). \]

\noindent\emph{Step 2. A homotopy construction.}
Note that $p_*^{-1}(U_j) = U_j \times_{\beta(X/\Ga)} \M$ is a compact Hausdorff fiber space over $U_j$ with convex fibers $\M_u = M_1(p^{-1}(u))$ for $u \in U_j$. Fix $s_0: U_j \to p_*^{-1}(U_j)$ to be the restriction of a continuous global section $\sigma: \beta(X/\Ga) \to \M$ of $p_*$ (whose existence will be proved in Step 3).

Because $F_j$ is a finite group acting continuously on $U_j$ and by affine transformations on the fibers of $p_*^{-1}(U_j)$, we can average $s_0$ over $F_j$ to obtain an $F_j$-equivariant continuous section $s_{F_j}: U_j \to p_*^{-1}(U_j)$ defined by
\[ s_{F_j}(u) := \frac{1}{|F_j|} \sum_{f \in F_j} f \cdot s_0(f^{-1} \cdot u), \qquad u \in U_j. \]
By the convexity of $\M_u$, $s_{F_j}(u)$ is a well-defined probability measure in $\M_u$. We then define the fiberwise affine homotopy $H: p_*^{-1}(U_j) \times [0, 1] \to p_*^{-1}(U_j)$ by
\[ H_t(\mu) := (1-t)\mu + t \, s_{F_j}\big(p_*(\mu)\big). \]
Since the $F_j$-action on the fibers of $p_*^{-1}(U_j)$ is affine and $s_{F_j}$ is $F_j$-equivariant, $H_t$ is continuous, preserves each fiber over $U_j$ and commutes strictly with the $F_j$-action. Moreover, $H_0 = \mathrm{id}_{p_*^{-1}(U_j)}$ and $H_1(\mu) = s_{F_j}(p_*(\mu))$ is the retraction onto the $F_j$-invariant image $s_{F_j}(U_j) \cong U_j$.

Because $H_t$ preserves the fibers over $U_j$, the induced map $H_t^*: C(p_*^{-1}(U_j)) \to C(p_*^{-1}(U_j))$ is a $C(U_j)$-linear, $\H_j$-equivariant $*$-homomorphism homotopy between $\mathrm{id}_{C(p_*^{-1}(U_j))}$ and $\iota_{U_j} \circ s_{F_j}^*$. Tensoring $H_t^*$ over $C(U_j)$ with $\mathrm{id}_{A_{U_j}}$ yields an $\H_j$-equivariant $*$-homomorphism homotopy on $A_{p_*^{-1}(U_j)}$, showing that $A_{U_j}$ and $A_{p_*^{-1}(U_j)}$ are $\H_j$-equivariantly homotopy equivalent. Therefore, $\iota_*^{U_j}$ is an isomorphism.

\noindent\emph{Step 3. Existence of a global section.}
To construct a continuous section $s_0$, we show that $p_*: \M \to \beta(X/\Ga)$ admits a global continuous section.

Since $\Gamma$ acts freely on the discrete space $X$, we can choose a fundamental domain $F \subseteq X$ containing exactly one representative from each right $\Gamma$-orbit, and let $\chi_F$ be its characteristic function. This defines a map $\sigma_0: X/\Gamma \to M_1(X) \subset M_1(\beta X)$ by assigning to each discrete orbit $y \in X/\Gamma$ the Dirac measure
\[ \sigma_0(y) = \sum_{z \in p^{-1}(y)} \chi_F(z) \delta_z = \delta_{z_0}, \]
where $z_0$ is the unique representative of the orbit $y$ in $F$. Since $X/\Gamma$ is discrete and $M_1(\beta X)$ is compact Hausdorff under the weak-$*$ topology, $\sigma_0$ extends uniquely to a continuous map $\sigma: \beta(X/\Gamma) \to M_1(\beta X)$ by the universal property of the Stone-\v{C}ech compactification. On the dense subset $X/\Gamma$, $p_* \circ \sigma_0$ is the identity map, so $p_* \circ \sigma = \mathrm{id}_{\beta(X/\Ga)}$ by continuity. Thus $\sigma$ takes values in $\M$ and defines a globally continuous section of $p_*: \M \to \beta(X/\Gamma)$. Restricting $\sigma$ to $U_j$ gives the desired section $s_0: U_j \to p_*^{-1}(U_j)$.
\end{proof}

\begin{proof}[Proof of \Cref{thm: ECNC for CE spaces}]
The theorem follows directly from the diagram \eqref{eq:iota} and \Cref{lem:iota}.
\end{proof}

\subsection*{Acknowledgements}
The author wishes to thank Jintao Deng and Jiawen Zhang for their numerous helpful suggestions during the early stages of this paper, and Kai Mao for helpful comments on an earlier version of the manuscript. Special thanks go to Jianchao Wu for an inspiring conversation that led the author to consider the problems addressed in this paper.

This work is partially supported by the National Natural Science Foundation of China (NSFC), Grant No. 12571135.

\subsection*{AI Declaration}
The mathematical ideas, theorems, and core proof strategies in this work were independently formulated by the author. During the revision process, Gemini 3.1 Pro was used to assist with improving the language and organization of the text, and Chat GPT 6.1 Sol was used to help identify inaccuracies in a preliminary draft. The author has verified all mathematical details and takes full responsibility for the content of this paper.

\bibliographystyle{alpha}
\bibliography{ref}
\end{document}